\documentclass[11pt]{amsart}

\usepackage{etex}
\usepackage{latexsym}
\usepackage{amssymb,amsfonts,amsmath,amsthm}
\usepackage{graphicx}
\usepackage{enumerate}
\usepackage{mdwlist}
\usepackage[margin=1.3in]{geometry}
\usepackage{pstricks,tikz}
\usepackage[all]{xy}

\usetikzlibrary{arrows}

\newcommand{\eps}{\varepsilon}

\newtheorem{theorem}{Theorem}  

\newtheorem{lemma}[theorem]{Lemma}

\newtheorem{corollary}[theorem]{Corollary}

\newtheorem{proposition}[theorem]{Proposition}

\newtheorem{question}[theorem]{Question}

 {\begin{list}{}%
         {\setlength{\leftmargin}{#1}}%
         \item[]%
 }
 {\end{list}}

\def\COMMENT#1{}
\def\TASK#1{}

\numberwithin{theorem}{section}
\numberwithin{equation}{section}

\newdimen\margin   
\def\textno#1&#2\par{%
   \margin=\hsize
   \advance\margin by -4\parindent
          \setbox1=\hbox{\sl#1}%
   \ifdim\wd1 < \margin
      $$\box1\eqno#2$$%
   \else
      \bigbreak
      \hbox to \hsize{\indent$\vcenter{\advance\hsize by -3\parindent
      \it\noindent#1}\hfil#2$}%
      \bigbreak
   \fi}
\def\noproof{{\unskip\nobreak\hfill\penalty50\hskip2em\hbox{}\nobreak\hfill%
       $\square$\parfillskip=0pt\finalhyphendemerits=0\par}\goodbreak}
\def\endproof{\noproof\bigskip}

\title[Forbidding induced even cycles in a graph: typical structure and counting]{Forbidding induced even cycles in a graph:\\ typical structure and counting}

\author{ Jaehoon Kim, Daniela K\"uhn, Deryk Osthus, Timothy Townsend }
\thanks{The research leading to these results was partially supported by the European Research Council
under the European Union's Seventh Framework Programme (FP/2007--2013) / ERC Grant
Agreements no. 258345 (D.~K\"uhn) and 306349 (J.~Kim and D.~Osthus).}

\begin{document}

\begin{abstract}
We determine, for all $k\geq 6$, the typical structure of graphs that do not contain an induced $2k$-cycle. This verifies a conjecture of Balogh and Butterfield. Surprisingly, the typical structure of such graphs is richer than that encountered in related results. The approach we take also yields an approximate result on the typical structure of graphs without an induced $8$-cycle or without an induced $10$-cycle.

\end{abstract}

\date{\today}

\maketitle

\section{Introduction}

\subsection{Background}

The enumeration and description of the typical structure of graphs with given side constraints has become a successful and popular area at the interface of probabilistic, enumerative, and extremal combinatorics (see e.g.~\cite{BolSurv} for a survey of such work). For example, a by now classical result of Erd\H{o}s, Kleitman and Rothschild~\cite{EKR} shows that almost all triangle-free graphs are bipartite (given a fixed graph $H$, a graph is called $H$-\emph{free} if it does not contain $H$ as a not necessarily induced subgraph). This result was generalised to $K_k$-free graphs by Kolaitis, Pr\"omel and Rothschild~\cite{KPR}. There are now many precise results on the number and typical structure of $H$-free graphs and more generally graphs, hypergraphs and other combinatorial structures with a given (anti-)monotone property.

Given a fixed graph $H$, a graph is called \emph{induced}-$H$-\emph{free} if it does not contain $H$ as an induced subgraph. Associated counting and structural questions are equally natural as in the non-induced case, but seem harder to solve. Thus much less is known about the typical structure and number of induced-$H$-free graphs than that of $H$-free graphs, though considerable work has been done in this area (see, e.g.~\cite{ABBM, BaBu, KMRS, PrSt2, PrSt3, PrSt4}). In particular, Pr\"omel and Steger~\cite{PrSt4} obtained an asymptotic counting result for the number of induced-$H$-free graphs on $n$ vertices, showing that the logarithm of this number is essentially determined by the so-called colouring number of $H$. This was generalised to arbitrary hereditary properties independently by Alekseev~\cite{Alek} as well as Bollob\'as and Thomason~\cite{BoTh}. Recent exciting developments in~\cite{BMS, SaTh} have opened up the opportunity to replace counting results by more precise results which identify the typical asymptotic structure.

In this paper we determine the typical structure of induced-$C_{2k}$-free graphs (from which the corresponding asymptotic counting result follows immediately). The key difficulty we encounter is that the typical structure turns out to be more complex than encountered in previous results on forbidden induced subgraphs. This requires new ideas and a more intricate analysis when `excluding' classes of graphs which might be candidates for typical induced-$C_{2k}$-free graphs.

\subsection{Graphs with forbidden induced cycles}

Given a class of graphs $\mathcal{A}$, we let $\mathcal{A}_n$ denote the set of all graphs in $\mathcal{A}$ that have precisely $n$ vertices, and we say that \emph{almost all graphs in} $\mathcal{A}$ \emph{have property} $\mathcal{B}$ if
$$\lim\limits_{ n\to \infty }\frac{|\{G\in \mathcal{A}_n: G \text{ has property } \mathcal{B}\}|}{|\mathcal{A}_n|} = 1.$$
Given graphs $H_1,\dots, H_m$, we say \emph{$G$ can be covered by $H_1,\dots, H_m$} if $V(G)$ admits a partition $A_1\cup \dots \cup A_m=V(G)$ such that $G[A_i]$ is isomorphic to $H_i$ for every $i\in \{1,\dots, m\}$.

Pr\"omel and Steger proved in~\cite{PrSt2} that almost all induced-$C_4$-free graphs can be covered by a clique and an independent set, and in~\cite{PrSt1} characterised the structure of almost all induced-$C_5$-free graphs too. More recently, Balogh and Butterfield~\cite{BaBu} determined the typical structure of induced-$H$-free graphs for a wide class of graphs $H$. In particular they proved that almost all induced-$C_7$-free graphs can be covered by either three cliques or two cliques and an independent set, and that for $k\geq 4$ almost all induced-$C_{2k+1}$-free graphs can be covered by $k$ cliques. They also conjectured that for $k\geq 6$ almost all induced-$C_{2k}$-free graphs can be covered by $k-2$ cliques and a graph whose complement is a disjoint union of stars and triangles. Our main result completely verifies this conjecture.

\begin{theorem}\label{main theorem}
For $k\geq 6$, almost all induced-$C_{2k}$-free graphs can be covered by $k-2$ cliques and a graph whose complement is a disjoint union of stars and triangles.
\end{theorem}

Theorem~\ref{main theorem} together with the discussed results in~\cite{BaBu, EKR, PrSt1, PrSt2} implies that the typical structure of induced-$C_k$-free graphs is determined for every $k\in \mathbb{N}$ apart from $k\in \{6,8,10\}$. For the cases $k=8$ and $k=10$ the methods we use to prove Theorem~\ref{main theorem} allow us to also prove an approximate result on the typical structure of induced-$C_k$-free graphs. In order to state this result we require the following definitions.

Given $\eta>0$ and graphs $G$ and $G'$ on the same vertex set, we say $G'$ is \emph{$\eta$-close to $G$} if $G'$ can be made into $G$ by changing (i.e. adding or deleting) at most $\eta |G|^2$ edges. We say a graph $G$ is a {\em sun} if either $G$ consists of a single vertex or $V(G)$ can be partitioned into sets $A,B$ such that $E(G)=\{uv : |\{u,v\}\cap B|\leq 1\}$. We call $A$ the {\em body} of the sun and $B$ the {\em side} of the sun. Note that all stars and cliques (including triangles) are suns, and that we consider a single vertex to be both a star of order one and a clique of order one.

\begin{theorem}\label{secondary theorem}~
\begin{enumerate}[{\rm (i)}]
\item For every $\eta>0$, almost all induced $C_{10}$-free graphs are $\eta$-close to graphs that can be covered by three cliques and a graph whose complement is a disjoint union of cliques.
\item For every $\eta>0$, almost all induced $C_{8}$-free graphs are $\eta$-close to graphs that can be covered by two cliques and a graph whose complement is a disjoint union of suns.
\end{enumerate}
\end{theorem}

We remark that in Theorems~\ref{main theorem} and~\ref{secondary theorem} we get exponential bounds on the proportion of induced-$C_{2k}$-free graphs that do not satisfy the relevant structural description. Our proofs also show that the $k-2$ cliques in the covering have size close to $n/(k-1)$ in Theorem~\ref{main theorem}, with analogous bounds in Theorem~\ref{secondary theorem}. Theorem~\ref{main theorem} also strengthens a result by Kang, McDiarmid, Reed and Scott~\cite{KMRS} showing that almost all induced-$C_{2k}$-free graphs have a linear sized homogeneous set. (Their results were motivated by the Erd\H{o}s-Hajnal conjecture, and actually apply to a large class of forbidden graphs $H$.)

It would of course be interesting to determine the typical structure of induced-$C_6$-free graphs.

\begin{question}
What is the typical structure of induced-$C_6$-free graphs?
\end{question}

It seems likely that almost all induced-$C_6$-free graphs can be covered by one clique and one cograph, where a cograph is a graph not containing an induced copy of $P_4$. Another natural question is that of the typical structure of induced-$H$-free graphs of a given density. In particular, an intriguing question is whether their typical structure exhibits a non-trivial `phase transition' as found for triangle-free graphs~\cite{OPT} and more generally $K_r$-free graphs~\cite{BMSW}.

\subsection{Overview of the paper}

A key tool in our proofs is the recent hypergraph container approach, which was developed independently by Balogh, Morris and Samotij~\cite{BMS}, and Saxton and Thomason~\cite{SaTh}. Briefly, their result states that under suitable conditions on a uniform hypergraph $G$, there is a small collection $\mathcal{C}$ of small subsets (known as containers) of $V(G)$ such that every independent set of vertices in $G$ is a subset of some element of $\mathcal{C}$. The precise statement of the application used here is deferred until Section~\ref{sec: rough structure}.

Given a graph $G$ and a set $A\subseteq V(G)$, we denote by $G[A]$ the graph induced on $G$ by $A$, and we denote the complement of $G$ by $\overline{G}$. For $k\in \mathbb{N}$ and a set $V$ of vertices we define an \emph{ordered $k$-partition of $V$} to be a $k$-partition of $V$ such that one partition class is labelled and the rest are unlabelled. If $Q$ is an ordered $k$-partition with labelled class $Q_0$ and unlabelled classes $Q_1,\dots, Q_{k-1}$ then we write $Q=(Q_0,\{Q_1,\dots, Q_{k-1}\})$.

For $k\geq 4$, we say that a graph $G$ is a {\em $k$-template} if $V(G)$ has an ordered $(k-1)$-partition $Q=(Q_0 , \{Q_1 , \dots, Q_{k-2} \})$ such that $G[Q_i]$ is a clique for all $i\in [k-2]$ and one of the following holds.
\begin{itemize}
\item $k=4$ and $\overline{G}[Q_0]$ is a disjoint union of suns.
\item $k=5$ and $\overline{G}[Q_0]$ is a disjoint union of stars and cliques.
\item $k\geq 6$ and $\overline{G}[Q_0]$ is a disjoint union of stars and triangles.
\end{itemize}
Clearly every $k$-template is induced-$C_{2k}$-free. If $V(G)$ has such an ordered $(k-1)$-partition $Q$, we say that \emph{$G$ is a $k$-template on $Q$}, or \emph{$G$ has ordered $(k-1)$-partition $Q$}. If $Q'$ is the (unordered) $(k-1)$-partition with the same partition classes as $Q$, we may also say that \emph{$G$ is a $k$-template on $Q'$}. Thus Theorem~\ref{main theorem} can be reformulated as:
$$\textit{`For $k\geq 6$, almost all induced $C_{2k}$-free graphs are $k$-templates.'}$$
Theorem~\ref{secondary theorem} can be similarly reformulated in terms of $4$- and $5$-templates. As mentioned earlier, the main difficulty in proving Theorem~\ref{main theorem} (compared to related results) is that typically $G[Q_0]$ is close to, but not quite, a complete graph. This makes it very difficult to rule out other similar classes of graphs as typical structures. To overcome this we use tools such as Ramsey's theorem to classify the graphs according to the neighbourhoods of certain vertices.

More precisely, our approach to proving our main result is as follows. Firstly, in Section~\ref{sec: rough structure} we use the hypergraph containers result discussed above to show that almost all induced-$C_{2k}$-free graphs are close to being a $k$-template, for every $k\geq 4$ (see Lemma~\ref{lem: rough structure}). Note that Lemma~\ref{lem: rough structure} immediately implies Theorem~\ref{secondary theorem}.

In Section~\ref{sec: number of templates} we prove upper and lower bounds on the number of $k$-templates on $n$ vertices (see Lemmas~\ref{size of template} and~\ref{the number of templates}). In Section~\ref{sec: set-up} we prove some preliminary results about graphs that are close to being a $k$-template.

In Section~\ref{sec: main proof} we state a key result which is a version of Theorem~\ref{main theorem} with respect to a given ordered $(k-1)$-partition (see Lemma~\ref{induction conclusion}) and use it together with Lemma~\ref{the number of templates} to derive Theorem~\ref{main theorem}. The remainder of the paper is devoted to proving Lemma~\ref{induction conclusion} via an inductive argument, which we introduce at the end of Section~\ref{sec: main proof}. This argument involves partitioning the class of graphs considered in Lemma~\ref{induction conclusion} into three `bad' classes of graphs, and in each of Sections~\ref{sec: estimate 1}, \ref{sec: estimate 2} and~\ref{sec: estimate 3} we use Lemma~\ref{size of template} and the results in Section~\ref{sec: set-up} to prove an upper bound on the number of graphs in a different one of these classes (see Lemmas~\ref{F^1}, \ref{F^2_Q} and~\ref{F^3}). In particular, Lemmas~\ref{F^1} and~\ref{F^2_Q} already show that almost all induced-$C_{2k}$-free graphs are `extremely close' to being $k$-templates (see Proposition~\ref{beta}). Finally in Section~\ref{sec: final calculation} we use Lemmas~\ref{lem: rough structure}, \ref{F^1}, \ref{F^2_Q} and~\ref{F^3} to complete the inductive argument set up in Section~\ref{sec: main proof} and so prove Lemma~\ref{induction conclusion}. Before starting on any of this however, we lay out some notation and set out some useful tools in Section~\ref{sec: notation}, below.

\section{Notation and tools}\label{sec: notation}

Given a graph $G$, a vertex $x\in V(G)$, and an ordered $(k-1)$-partition $Q=(Q_0,\{Q_1,\dots,Q_{k-2}\})$ of $V(G)$, we let $N(x), \overline{N}(x)$ denote the set of neighbours and non-neighbours of $x$ in $G$, respectively. We also let $N_{Q_i}(x), \overline{N}_{Q_i}(x)$ denote the set of neighbours of $x$ in $Q_i$ and non-neighbours of $x$ in $Q_i$, respectively. We sometimes use the notation $d^i_{G,Q}(x)=|N_{Q_i}(x)|$ and $\overline{d}^i_{G,Q}(x)=|\overline{N}_{Q_i}(x)|$ when we want to emphasise which graph we are working with. For a set $A$ of vertices in $G$, we define
$$N(A):= \bigcap_{v\in A} N(v), \hspace{0.6cm} \overline{N}(A):= \bigcap_{v\in A} \overline{N}(v),$$
$$N_{Q_i}(A):= \bigcap_{v\in A} N_{Q_i}(v), \hspace{0.6cm} \overline{N}_{Q_i}(A):= \bigcap_{v\in A} \overline{N}_{Q_i}(v).$$
If it generates no ambiguity, we may write $N_i(x), \overline{N}_{i}(x), N_i(A), \overline{N}_i(A)$ for $N_{Q_i}(x), \overline{N}_{Q_i}(x), N_{Q_i}(A),$ and $\overline{N}_{Q_i}(A)$ respectively.
Given $A,B\subseteq V(G)$, we define
$$N^*(A,B) := N(A)\cap \overline{N}(B), \hspace{0.6cm} N_i^*(A,B) := N_i(A)\cap \overline{N}_i(B).$$
In the case when $A$ and $B$ both have size one, containing vertices $a,b$ respectively, we may write $N^*(a,b)$ for $N^*(A,B)$ and $N_i^*(a,b)$ for $N_i^*(A,B)$.

We say that a partition of vertices is \emph{balanced} if the sizes of any two partition classes differ by at most one. Given a $(k-1)$-partition $Q$ of $[n]$ with partition classes $Q_0,\dots,Q_{k-2}$, and a graph
$G = (V,E)$ on vertex set $[n]$, and an edge or non-edge $e = uv$ with $u \in Q_i$ and $v \in Q_j$, we call $e$ {\em crossing} if $i\neq j$ and {\em internal} if $i = j$.

We denote a path on $m$ vertices by $P_m$. Given a path $P=p_1\dots p_m$ and a sequence $A_1,\dots, A_m$ of sets of vertices, we say that $P$ \emph{has type} $A_1,\dots, A_m$ if $p_{\ell}\in A_{\ell}$ for every $\ell\in [m]$. We call a graph a \emph{linear forest} if it is a forest such that all components are paths or isolated vertices.

Given $\ell, t\in \mathbb{N}$ we let $R_{\ell}(t)$ denote the $\ell$-colour Ramsey number for monochromatic $t$-cliques, i.e.~$R_{\ell}(t)$ is the smallest $N\in \mathbb{N}$ such that every $\ell$-colouring of the edges of $K_{N}$ yields a monochromatic copy of $K_t$.

We define
$$n_k:= \left\lceil \frac{n}{k-1} \right\rceil.$$

In a number of our proofs we shall use the following Chernoff bound.
\begin{lemma}[Chernoff bound]\label{Chernoff Bounds}
Let $X$ have binomial distribution and let $0<a\leq \mathbb{E}[X]$. Then
\begin{enumerate}[{\rm (i)}]
\item $P(X>\mathbb{E}[X]+a)\leq \exp \left( -\frac{a^2}{4\mathbb{E}[X]} \right)$.
\item $P(X<\mathbb{E}[X]-a)\leq \exp \left( -\frac{a^2}{2\mathbb{E}[X]} \right)$.
\end{enumerate}
\end{lemma}

Whenever this does not affect the argument, we assume all large numbers to be integers, so that we may sometimes omit floors and ceilings for the sake of clarity. In some proofs, given $a,b\in \mathbb{R}$ with $0<a,b<1$, we will use the notation $a\ll b$ to mean that we can find an increasing function $g$ for which all of the conditions in the proof are satisfied whenever $a\leq g(b)$. Throughout we write $\log x$ to mean $\log_2 x$.

We define $\xi(p):= - 3 p(\log p )/2$. The following bounds will prove useful to us. For $n\geq 1$ and $3\log n /n\leq p \leq 10^{-11}$,
\begin{equation}\label{entropy bound}
\binom{n}{\leq pn}:= \sum\limits_{i=0}^{\left\lfloor pn \right\rfloor} \binom{n}{i} \leq pn \left( \frac{en}{pn} \right)^{pn}\leq 2^{\xi(p)n},
\end{equation}
and
\begin{equation}\label{entropy bound 1.5}
\xi(p)\leq \frac{3}{2}p \left(\frac{1}{p} \right)^{1/8}\leq p^{3/4}.
\end{equation}

\section{Approximate structure of typical induced-$C_{2k}$-free graphs}\label{sec: rough structure}

The main result of this section is Lemma~\ref{lem: rough structure}, which approximately determines the typical structure of induced-$C_{2k}$-free graphs. As mentioned earlier, we make use of a `container theorem' which reduces the proof of Lemma~\ref{lem: rough structure} to an extremal problem involving induced-$C_{2k}$-free graphs. More precisely, the argument is structured as follows.

We first introduce a number of tools (see Subsection~\ref{subsec: rough 1}): a `Containers' theorem (Theorem~\ref{thm: containers}), a Stability theorem (Theorem~\ref{thm: basic stability}), and two Removal Lemmas (Theorem~\ref{thm: induced removal lemma}, Lemma~\ref{lem: multigraph removal lemma}). In Subsection~\ref{subsec: rough 2} we use Theorem~\ref{thm: basic stability} to derive a Stability result involving induced-$C_{2k}$-free graphs (Lemma~\ref{lem: specialised stability}). Similarly we use Theorem~\ref{thm: induced removal lemma} to derive another specialised version of the Removal Lemma (Lemma~\ref{lem: specialised removal lemma}). In Subsection~\ref{subsec: rough 3} we use Theorem~\ref{thm: containers} together with Lemmas~\ref{lem: multigraph removal lemma},~\ref{lem: specialised stability} and~\ref{lem: specialised removal lemma} to determine the approximate structure of typical induced-$C_{2k}$-free graphs.

We denote the number of (labelled) induced-$C_{2k}$-free graphs on $n$ vertices by $F(n,k)$.

\begin{lemma}\label{lem: rough structure}
Let $k\geq 4$. For every $\eta >0$ there exists $\varepsilon >0$ such that the following holds for all sufficiently large $n$. All but at most $F(n,k) 2^{-\varepsilon n^2}$ induced-$C_{2k}$-free graphs on $n$ vertices can be made into a $k$-template by changing at most $\eta n^2$ edges.
\end{lemma}

Note that Lemma~\ref{lem: rough structure} immediately implies Theorem~\ref{secondary theorem}.

\subsection{Tools: containers, stability and removal lemmas}\label{subsec: rough 1}

The key tool in this section is Theorem~\ref{thm: containers}, which is an application of the more general theory of Hypergraph Containers developed in~\cite{BMS,SaTh}. We use the formulation of Theorem~1.5 in~\cite{SaTh}. We require the following definitions in order to state it.

A \emph{$2$-coloured multigraph} $G$ on vertex set $[N]$ is a pair of edge sets $G_R, G_B\subseteq [N]^{(2)}$, which we call the red and blue edge sets respectively. If $H$ is a fixed graph on vertex set $[h]$, a copy of $H$ in $G$ is an injection $f:[h]\rightarrow[N]$ such that for every edge $uv$ of $H$, $f(u)f(v)\in G_R$, and for every non-edge $u'v'$ of $H$, $f(u')f(v')\in G_B$. We write $H\subseteq G$ if $G$ contains a copy of $H$, and we say that $G$ is \emph{$H$-free} if there are no copies of $H$ in $G$. We say that $G$ is \emph{complete} if $G_R \cup G_B = [N]^{(2)}$. We denote by $G^B$ the graph on vertex set $[N]$ and edge set $G_B$.

\begin{theorem}\label{thm: containers}
Let $H$ be a fixed graph with $h:= |V(H)|$. For every $\eps>0$, there exists $c>0$ such that for all sufficiently large $N$, there exists a collection $\mathcal{C}$ of complete $2$-coloured multigraphs on vertex set $[N]$ with the following properties.
\begin{enumerate}[{\rm (a)}]
\item For every graph $I$ on $[N]$ that contains no induced copy of $H$, there exists $G\in \mathcal{C}$ such that $I \subseteq G$.
\item Every $G\in \mathcal{C}$ contains at most $\eps N^h$ copies of $H$.
\item $\log |\mathcal{C}| \leq c N^{2 - (h-2)/(\binom{h}{2}-1)} \log N$.
\end{enumerate}
\end{theorem}

Another tool that we will use is the following classical Stability theorem of Erd\H{o}s and Simonovits (see e.g.~\cite{Erd1, Erd2, Sim}). By $T_k(n)$ we denote the \emph{Tur\'an graph}, the largest complete $k$-partite graph on $n$ vertices, and we define $t_{k}(n):= e(T_k(n))$. Given a family $\mathcal{H}$ of fixed graphs, we say a graph $G$ is \emph{$\mathcal{H}$-free} if $G$ does not contain any $H\in \mathcal{H}$ as a (not necessarily induced) subgraph, and we say $G$ is \emph{induced-$\mathcal{H}$-free} if $G$ does not contain any $H\in \mathcal{H}$ as an induced subgraph.

\begin{theorem}\label{thm: basic stability}
Let $\mathcal{H}=\{H_1,\dots, H_{\ell}\}$ be a family of fixed graphs, and let $k:= \min_{1\leq i\leq \ell} \chi (H_i)$. For every $\delta>0$ there exists $\varepsilon>0$ such that the following holds for all sufficiently large $n$. If a graph $G$ on $n$ vertices is $\mathcal{H}$-free and $e(G)\geq t_{k-1}(n)-\varepsilon n^2$, then $G$ can be obtained from $T_{k-1}(n)$ by changing at most $\delta n^2$ edges.
\end{theorem}

The final tools that we introduce in this subsection are the following two Removal Lemmas. The first is an extension of the Induced Removal Lemma to families of forbidden graphs, and is due to Alon and Shapira~\cite{AlSh}. The original statement of this theorem also applies to infinite families of forbidden graphs, but the version for finite families is sufficient for our purposes. The second is a version of the Removal Lemma applicable to complete $2$-coloured multigraphs. The proof is similar to that of the standard Removal Lemma, so we omit it here; for details see~\cite{TownsendPhD}. For two sets $A,B$, we denote their symmetric difference by $A\triangle B$. For $2$-coloured multigraphs $G, G'$ on the same vertex set we define their distance by $\text{dist}(G, G'):= |G_R\triangle G'_R| + |G_B\triangle G'_B|$.

\begin{theorem}\cite{AlSh}\label{thm: induced removal lemma}
For every finite family of fixed graphs $\mathcal{H}$ and every $\delta >0$, there exists $\varepsilon>0$ such that the following holds for all sufficiently large $n$. If a graph $G$ on $n$ vertices contains at most $\varepsilon n^h$ induced copies of each graph $H\in \mathcal{H}$ on $h$ vertices, then $G$ can be made induced-$\mathcal{H}$-free by changing at most $\delta n^2$ edges.
\end{theorem}

\begin{lemma}\label{lem: multigraph removal lemma}
For every fixed graph $H$ on $h$ vertices, and every $\delta>0$, there exists $\varepsilon>0$ such that the following holds for all sufficiently large $n$. If a complete $2$-coloured multigraph $G$ on vertex set $[n]$ contains at most $\varepsilon n^h$ copies of $H$, then there exists a complete $2$-coloured multigraph $G'$ on vertex set $[n]$ such that $G'$ is $H$-free and ${\rm dist}(G, G')\leq \delta n^2$.
\end{lemma}

\subsection{Stability and removal lemmas for even cycles}\label{subsec: rough 2}

Suppose $H$ is a complete $2$-coloured multigraph on $m$ vertices with $H_R\cap H_B=\emptyset$. If $m=3$ and $|H_R|\leq 1$ we call $H$ a \emph{mostly blue triangle}. For $k\in \{4,5,6\}$, if $m=4$ and $|H_R|\geq 6-k$ and $H^B$ contains a copy of $P_4$ then we call $H$ a \emph{$k$-good tetrahedron}. The following technical proposition will be useful in proving Lemmas~\ref{lem: specialised stability} and~\ref{lem: specialised removal lemma}.

\begin{proposition}\label{prop: C_2k in small multigraphs}
Let $k\geq 4$ and let $G$ be a complete $2$-coloured multigraph on $2k$ vertices. If $G$ satisfies one of the following properties then $G$ contains a copy\COMMENT{Note that I do want to say copy, instead of induced copy, here and in similar instances throughout this section, since $G$ is a $2$-coloured multigraph, and this useage is consistent with the definition of $G$ containing a fixed graph that was given at the beginning of the section.} of $C_{2k}$. Below, $r_i$ always denotes a red edge.
\begin{enumerate}[{\rm (E1)}]
\item $G_R\triangle G_B$ is a set of at most $k$ disjoint (red or blue) edges.
\item $G_R\triangle G_B$ is the edge set of two disjoint copies of a blue $K_k$.
\item $G_R\triangle G_B$ is the edge set of a union of disjoint graphs $K_3^1, K_3^2, r_1,\dots, r_{k-3}$, where each $K_3^i$ is a mostly blue triangle\COMMENT{Note that (E3) is not sufficient to imply that $G$ contains a $C_{2k}$ in the case $k=3$ (in particular, there is no covering permutation if exactly one of the mostly blue triangles contains a red edge). This creates big problems when trying to emulate the proof of Lemma~\ref{lem: specialised removal lemma} for $k=3$. So extending the rough structure section to the $k=3$ case may be more effort than it's worth. - Tim}.
\item $G_R\triangle G_B$ is the edge set of a union of disjoint graphs $K_4^1, r_1,\dots, r_{k-2}$, where $K_4^1$ is a $4$-good tetrahedron.
\item $k\geq 5$ and $G_R\triangle G_B$ is the edge set of a union of disjoint graphs $K_4^1, r_1,\dots, r_{k-2}$, where $K_4^1$ is a $5$-good tetrahedron.
\item $k\geq 6$ and $G_R\triangle G_B$ is the edge set of a union of disjoint graphs $K_4^1, r_1,\dots, r_{k-2}$, where $K_4^1$ is a $6$-good tetrahedron.
\end{enumerate}
\end{proposition}
\begin{proof}
Let $V(G)=\{v_1,\dots, v_{2k}\}$. Let $C=c_1\dots c_{2k}$ be a $2k$-cycle. Note that if there exists a permutation $\sigma$ of $[2k]$ such that for every edge $c_ic_j\in E(C)$ we have $v_{\sigma(i)}v_{\sigma(j)}\in G_R$ and such that for every non-edge $c_{i'}c_{j'}\notin E(C)$ we have $v_{\sigma(i')}v_{\sigma(j')}\in G_B$, then $v_{\sigma(1)}\dots v_{\sigma(2k)}$ is a copy of $C_{2k}$ in $G$. We call such a permutation $\sigma$ a \emph{covering permutation from $C$ to $G$}.
For ease of reading, we will write a permutation $\sigma$ on $[2k]$ using the notation $\sigma = (\sigma(1), \dots, \sigma(2k))$. If $\sigma$ restricted to $\{m, m+1, \dots, 2k\}$ is the identity permutation, we may simply write $\sigma=\{\sigma(1),\dots, \sigma(m-1)\}$ instead. So for example if $\sigma=(1,3,4,2)$ is a covering permutation from $C$ to $G$, then $v_1v_3v_4v_2v_5 \dots v_{2k}$ is a copy of $C_{2k}$ in $G$.

We now show that each of the properties (E1),$\dots$,(E6) imply that there exists a covering permutation from $C$ to $G$, and hence that $G$ contains a copy of $C_{2k}$.

\begin{enumerate}[{\rm (E1)}]
\item There exists $b,r\in \mathbb{N}\cup \{0\}$ with $b+r\leq k$ such that, by relabelling vertices if necessary, $G_B \backslash G_R =\{v_1v_2, \dots, v_{2b-1}v_{2b}\}$ and $G_R \backslash G_B=\{v_{2b+1}v_{2b+2}, \dots, v_{2(b+r) -1}v_{ 2(b+r)}\}$. Depending on the value of $b$ we find the following covering permutations $\sigma$ from $C$ to $G$, as required.
\begin{itemize}
\item If $b=0$ then $\sigma$ is the identity permutation.
\item If $b=1$ then $\sigma=(1,3,4,2)$.
\item If $b\geq 2$ then $\sigma=(1,3,\dots, 2b-1, 2,4, \dots, 2b).$
\end{itemize}

\item Let $\{v_1,\dots, v_k\}, \{v_{k+1},\dots, v_{2k}\}$ be the respective vertex sets of the two copies of a blue $K_k$ in $G_R\triangle G_B$. Then $\sigma=(1,k+1,2,k+2,\dots,k,2k)$ is a covering permutation from $C$ to $G$, as required.

\item Let $V(K_3^1)=\{v_1,v_2,v_3\}, V(K_3^2)=\{v_4,v_5,v_6\}$ and $V(r_i)=\{v_{2i+5}, v_{2i+6}\}$ for every $i\in [k-3]$. Depending on the colour of the edges in $K_3^1, K_3^2$ we find the following covering permutations $\sigma$ from $C$ to $G$, as required.
\begin{itemize}
\item If $K_3^1, K_3^2$ both contain no red edges, then $\sigma = (1,4,2,5,3,6)$.
\item If $K_3^1$ contains exactly one red edge $v_1v_2$ and $K_3^2$ contains no red edges, then $\sigma = (4,1,2,5,3,6)$.
\item If $K_3^1$ contains exactly one red edge $v_1v_2$ and $K_3^2$ contains exactly one red edge $v_5v_6$, then $\sigma = (1,2,4,3,5,6)$.
\end{itemize}

\item Let $V(K_4^1)=\{v_1,v_2,v_3,v_4\}$ and $V(r_i)=\{v_{2i+3}, v_{2i+4}\}$ for every $i\in [k-2]$. Depending on the configuration of red edges in $K_4^1$ we find the following covering permutations $\sigma$ from $C$ to $G$, as required.
\begin{itemize}
\item If $K_4^1$ contains exactly three red edges $v_1v_2, v_2v_3, v_3v_4$, then $\sigma$ is the identity permutation.
\item If $K_4^1$ contains exactly two red edges $v_1v_2, v_2v_3$, then $\sigma = (1,2,3,5,6,4)$.
\item If $K_4^1$ contains exactly two red edges $v_1v_2, v_3v_4$, then $\sigma = (1,2,5,6,3,4)$.
\end{itemize}

\item We may assume that $K_4^1$ contains exactly one red edge, since that is the only case not covered by (E4). Let $V(K_4^1)=\{v_1,v_2,v_3,v_4\}$ and $V(r_i)=\{v_{2i+3}, v_{2i+4}\}$ for every $i\in [k-2]$, and let $v_1v_2$ be the red edge in $K_4^1$. Then $\sigma = (1,2,5,6,3,7,8,4)$ is a covering permutation from $C$ to $G$, as required.

\item We may assume that $K_4^1$ contains no red edges, since that is the only case not covered by (E5). Let $V(K_4^1)=\{v_1,v_2,v_3,v_4\}$ and $V(r_i)=\{v_{2i+3}, v_{2i+4}\}$ for every $i\in [k-2]$. Then $\sigma = (1,5,6,2,7,8,3,9,10,4)$ is a covering permutation from $C$ to $G$, as required.
\end{enumerate}
\end{proof}

We now use Theorem~\ref{thm: basic stability} and Proposition~\ref{prop: C_2k in small multigraphs} to prove the following more specialised Stability result involving $C_{2k}$-free $2$-coloured multigraphs.

\begin{lemma}\label{lem: specialised stability}
Let $k\geq 4$. For every $\delta>0$ there exists $\varepsilon>0$ such that the following holds for all sufficiently large $n$. If a complete $2$-coloured multigraph $G$ on vertex set $[n]$ is $C_{2k}$-free and $|G_R\cap G_B|\geq t_{k-1}(n)-\varepsilon n^2$, then the graph $([n], G_R\cap G_B)$ can be obtained from $T_{k-1}(n)$ by changing at most $\delta n^2$ edges.
\end{lemma}
\begin{proof}
Choose $n_0\in \mathbb{N}$ and $\varepsilon>0$ such that $1/n_0 \ll \varepsilon \ll \delta$. Let $n\geq n_0$. Since $G$ is $C_{2k}$-free, we know by Proposition~\ref{prop: C_2k in small multigraphs} that no $2k$ vertices of $G$ induce on $G$ a $2$-coloured multigraph $G'$ that satisfies (E1). So, since $G$ is complete, the graph $([n], G_R\cap G_B)$ must be $T_k(2k)$-free. Note that $\chi(T_k(2k))=k$. By Theorem~\ref{thm: basic stability}, this together with the fact that $|G_R\cap G_B|\geq t_{k-1}(n)-\varepsilon n^2$ implies that the graph $([n], G_R\cap G_B)$ can be obtained from $T_{k-1}(n)$ by changing at most $\delta n^2$ edges.
\end{proof}

The following proposition characterises the structure of graphs without $k$-good tetrahedrons. It will be useful in proving Lemma~\ref{lem: specialised removal lemma}. The proof is fairly straightforward so we give only a sketch of it here.

\begin{proposition}\label{prop: k-good tetradehron free}
Let $G$ be a $2$-coloured multigraph with $G_R\cap G_B=\emptyset$.
\begin{enumerate}[{\rm (i)}]
\item If $G$ does not contain a $6$-good tetrahedron then $G^B$ is a disjoint union of stars and triangles\COMMENT{Full Proof: Suppose a $2$-coloured multigraph $G$ is $6$-good tetrahedron-free. Then $G^B$ is $P_4$-free, by definition. So all paths in $G^B$ have length at most $2$. Let $H$ be a component of $G^B$ that is not a triangle. Then $H$ is a tree, since any $4$-cycle or longer in $H$ would contain a $P_4$ (as would a triangle with a pendant edge). Suppose for a contradiction that $H$ has $2$ vertices $u,v$ both of degree at least $2$. Find a path $P=\{u,u_1,\dots, u_m,v\}$ in $H$ (of length at least $1$). Note that if the path has length $1$ then $u=u_m$ and $v=u_{1}$. Note that since $H$ is a tree, $N(u)\backslash \{u_1\}\cap P=\emptyset$ and $N(v)\backslash \{u_m\}\cap P=\emptyset$. So since $|N(u)|, |N(v)|\geq 2$, there exist vertices $u', v'$ such that $u'\in N(u)\backslash P$ and $v'\in N(v)\backslash P$. Since $H$ is a tree, $u'\ne v'$, and $u'Pv'$ is a path of length at least $3$, which contradicts the fact that $H$ is $P_4$-free. So $H$ has at most one vertex of degree at least $2$, and so $H$ must be a star. So $G^B$ is a disjoint collection of stars and triangles, as required.}.
\item If $G$ does not contain a $5$-good tetrahedron then $G^B$ is a disjoint union of stars and cliques\COMMENT{Full Proof: Suppose a $2$-coloured multigraph $G$ is $5$-good tetrahedron-free. Then the vertices of every copy of $P_4$ in $G^B$ induce a $K_4$, by definition. Let $H$ be a component of $G^B$ that is not a star or triangle. Since $P_4$-free graphs consist only of a disjoint union of stars and triangles, $H$ must contain a $P_4$, and hence a $K_4$. We now prove that $H$ is a clique by induction on $n=|H|$, for all $n\geq 4$. For the base case, $n=4$ so $H=K_4$ and we are done. Otherwise, suppose $n> 4$. Choose a $K_4$ in $H$ and choose $1$ vertex $v$ not in this $K_4$, and note that $H-v$ is not a triangle or a star (since it contains a $K_4$). So by the inductive hypothesis, $H-v$ is a clique. Suppose for a contradiction that there is a vertex $u$ in $H-v$ such that $uv\notin E(H)$. Then $v$ and $u$ together with a vertex in $H-v$ adjacent to $v$, and any other vertex in $H-v$ forms a $P_4$ but not a $K_4$. This is a contradiction, and hence $H$ is a clique. This completes the induction, and hence the proof.}.
\item If $G$ does not contain a $4$-good tetrahedron then $G^B$ is a disjoint union of suns.
\end{enumerate}
\end{proposition}
\begin{proof}
(i) follows immediately from the fact that if $G$ is $6$-good tetrahedron-free then $G^B$ does not contain a $P_4$.

To see (ii), note that if $G$ is $5$-good tetrahedron-free and $P$ is a copy of $P_4$ in $G^B$, then $G^B[V(P)]=K_4$. So every component $H$ of $G^B$ is either a star or a triangle or contains a $K_4$. But in the latter case it is easy to check that $H$ is actually a clique.

It remains to prove (iii). If $G$ is $4$-good tetrahedron-free and $P$ is a copy of $P_4$ in $G^B$, then $G^B[V(P)]$ is either a $K_4$ or a copy of the graph $K_4^-$ obtained from $K_4$ by deleting one edge. So every component $H$ of $G^B$ is either a star or a clique or contains an induced copy of $K_4^-$. Using induction on $|H|$, it is not hard to show that in the latter case $H$ must be a sun.
\end{proof}

We now use Theorem~\ref{thm: induced removal lemma} together with Propositions~\ref{prop: C_2k in small multigraphs} and~\ref{prop: k-good tetradehron free} to prove the following more specialised Removal Lemma involving even cycles.

\begin{lemma}\label{lem: specialised removal lemma}
For every $k\geq 4$ and every $\delta>0$ there exists $\varepsilon>0$ such that the following holds for all sufficiently large $n$. Suppose $G$ is a complete $2$-coloured multigraph on $n$ vertices such that $G_R\cap G_B = E(T_{k-1}(n))$. Let $Q$ be the unique $(k-1)$-partition of the vertices of $G$ such that no partition class induces an edge in $G_R\cap G_B$. Suppose further that $G$ contains at most $\varepsilon n^{2k}$ copies of $C_{2k}$. Then there exists a $k$-template $T=(V(G), E^T)$ on $Q$ such that $|G_R\triangle E^T|\leq \delta n^2$.
\end{lemma}
\begin{proof}
We first prove the lemma in the case $k\geq 6$. Choose $n_0\in \mathbb{N}$ and $\varepsilon, \gamma>0$ such that $1/n_0\ll \varepsilon \ll \gamma \ll \delta, 1/k$. Let $n\geq n_0$ and let $Q=(Q_1, \dots, Q_{k-1})$. Let $c:= \varepsilon^{1/3}$.

We claim that for no two distinct $i,j\in [k-1]$ do $G[Q_i]$ and $G[Q_j]$ both contain at least $c n^k$ copies of a blue $K_k$. Indeed, if they do then there are at least $c^2 n^{2k} > \varepsilon n^{2k}$ sets of $2k$ vertices that each induce on $G$ a $2$-coloured multigraph $G'$ that satisfies (E2). By Proposition~\ref{prop: C_2k in small multigraphs} each such $G'$ contains a copy of $C_{2k}$. This contradicts the assumption that $G$ contains at most $\varepsilon n^{2k}$ copies of $C_{2k}$, which proves the claim.

Thus there exists $J\subseteq [k-1]$ with $|J|\leq 1$ such that for all $i\in [k-1]$ with $i\notin J$, $G[Q_i]$ contains fewer than $c n^k$ copies of a blue $K_k$. Together with Theorem~\ref{thm: induced removal lemma} (applied to $G^B[Q_i]$) this implies that $G[Q_i]$ can be made free of blue cliques of size $k$ by changing the colour of at most $\gamma n^2$ edges inside $Q_i$. So by Tur\'an's Theorem, for all $i\in [k-1]$ with $i\notin J$, $G[Q_i]$ must have at least
$$(k-1)\binom{n/(k-1)^2}{2}-2\gamma n^2 \geq \frac{n^2}{4(k-1)^3}$$
red edges.

\vspace{0.3cm}
\noindent {\bf Claim 1:} \emph{There is at most one index $i\in [k-1]$ such that $G[Q_i]$ contains at least $c n^3$ mostly blue triangles. Moreover, if there is such an index $i$ then $J\subseteq \{i\}$, and if there is no such index then $J=\emptyset$.}

\noindent Indeed, suppose for a contradiction that there exist distinct $i,j\in [k-1]$ such that $Q_i, Q_j$ both contain at least $c n^3$ mostly blue triangles. Note that any class that contains at least $c n^k$ copies of a blue $K_k$ must contain at least $c n^3$ mostly blue triangles. So we may assume that $J\subseteq \{i,j\}$. Thus for every index $\ell \ne i,j$, $G[Q_{\ell}]$ contains at least $n^2/(4(k-1)^3)$ red edges. Thus there are at least $2\varepsilon n^{2k}$ sets of $2k$ vertices that each induce on $G$ a $2$-coloured multigraph $G'$ that satisfies (E3). (To see this, note that to choose such a set of $2k$ vertices we may choose, for both indices $i,j$, the vertices of any one of the at least $c n^3$ mostly blue triangles in $G[Q_i], G[Q_j]$ respectively, and then choose, for each index $\ell \ne i,j$, any one of the at least $n^2/(4(k-1)^3)$ red edges in $Q_{\ell}$.) By Proposition~\ref{prop: C_2k in small multigraphs} each such $G'$ contains a copy of $C_{2k}$. This contradicts the assumption that $G$ contains at most $\varepsilon n^{2k}$ copies of $C_{2k}$, which proves the claim.
\vspace{0.3cm}

Let $J'$ consist of the index $j_0\in [k-1]$ such that $G[Q_{j_0}]$ contains at least $c n^3$ mostly blue triangles, if such an index exists. Otherwise let $J':= \emptyset$. Thus $J\subseteq J'$. For all $i\in [k-1]$ with $i\notin J'$, Claim~1 together with Theorem~\ref{thm: induced removal lemma} (applied to $G^B[Q_i]$) implies that $G[Q_i]$ can be made free of mostly blue triangles by changing the colour of at most $\gamma n^2$ edges inside $Q_i$. This implies that the blue edges inside $Q_i$ after such a change form a matching. Hence $G[Q_i]$ contains at most $2\gamma n^2$ blue edges.

If $J'=\emptyset$ then $G[Q_i]$ contains at most $2\gamma n^2$ blue edges for all $i\in [k-1]$, and hence $|G_B\backslash G_R|\leq \delta n^2$ (since $\gamma \ll \delta, 1/k$). In this case we are done by setting $T$ to be $K_n$. Otherwise, $J'=\{j_0\}$ and it suffices to show that the blue edges in $G[Q_{j_0}]$ can be made into the edge set of a disjoint collection of stars and triangles by changing the colour of at most $\gamma n^2$ edges inside $Q_{j_0}$, since then we are done by setting $T$ to be $K_n$ minus this disjoint collection of stars and triangles.

\vspace{0.3cm}
\noindent {\bf Claim 2(a):} \emph{$G[Q_{j_0}]$ contains fewer than $c n^4$ $6$-good tetrahedrons.}

\noindent Indeed, otherwise there are at least $\varepsilon^{1/2} n^{2k}$ sets of $2k$ vertices that each induce on $G$ a $2$-coloured multigraph $G'$ that satisfies (E6). (To see this, note that to choose such a set of $2k$ vertices we may first choose the vertices of any one of the at least $c n^4$ $6$-good tetrahedrons, and then choose, for each other class $Q_i$, any one of the at least $n^2/(4(k-1)^3)$ red edges in $Q_i$.) By Proposition~\ref{prop: C_2k in small multigraphs} each such $G'$ contains a copy of $C_{2k}$. This contradicts the assumption that $G$ contains at most $\varepsilon n^{2k}$ copies of $C_{2k}$, which proves the claim.
\vspace{0.3cm}

Claim~2(a) together with Theorem~\ref{thm: induced removal lemma} (applied to $G^B[Q_{j_0}]$) implies that $G[Q_{j_0}]$ can be made free of $6$-good tetrahedrons by changing the colour of at most $\gamma n^2$ edges inside $Q_{j_0}$. Proposition~\ref{prop: k-good tetradehron free}{\rm (i)} implies that after such a change, all blue edges inside $Q_{j_0}$ form a disjoint collection of stars and triangles, as required. This completes the proof in the case $k\geq 6$.

For the case $k=5$, the proof is almost identical to the case $k\geq 6$, except that instead of Claim~2(a) we prove the following weaker claim, which follows in a similar way.

\vspace{0.3cm}
\noindent {\bf Claim 2(b):} \emph{$G[Q_{j_0}]$ contains fewer than $c n^4$ $5$-good tetrahedrons.}
\vspace{0.3cm}

Claim~2(b) together with Theorem~\ref{thm: induced removal lemma} (applied to $G^B[Q_{j_0}]$) implies that $G[Q_{j_0}]$ can be made free of $5$-good tetrahedrons by changing the colour of at most $\gamma n^2$ edges inside $Q_{j_0}$. Proposition~\ref{prop: k-good tetradehron free}{\rm (ii)} implies that after such a change, all blue edges inside $Q_{j_0}$ form a disjoint collection of stars and cliques. We are now done by setting $T$ to be $K_n$ minus this disjoint collection of stars and cliques.

For the case $k=4$, the proof is again almost identical to the case $k\geq 6$, except that instead of Claim~2(a) we prove the following even weaker claim, which follows in a similar way.

\vspace{0.3cm}
\noindent {\bf Claim 2(c):} \emph{$G[Q_{j_0}]$ contains fewer than $c n^4$ $4$-good tetrahedrons.}
\vspace{0.3cm}

Claim~2(c) together with Theorem~\ref{thm: induced removal lemma} (applied to $G^B[Q_{j_0}]$) implies that $G[Q_{j_0}]$ can be made free of $4$-good tetrahedrons by changing the colour of at most $\gamma n^2$ edges inside $Q_{j_0}$. Proposition~\ref{prop: k-good tetradehron free}{\rm (iii)} implies that after such a change, all blue edges inside $Q_{j_0}$ form a disjoint collection of suns. We are now done by setting $T$ to be $K_n$ minus this disjoint collection of suns.
\end{proof}

\subsection{Approximate structure of typical induced $C_{2k}$-free graphs}\label{subsec: rough 3}

We are now in a position to prove the main result of this section.

\removelastskip\penalty55\medskip\noindent{\bf Proof of Lemma~\ref{lem: rough structure}.}
Choose $n_0\in \mathbb{N}$ and $\varepsilon, \delta, \gamma, \beta>0$ such that $1/n_0\ll \varepsilon \ll \delta \ll \gamma \ll \beta \ll \eta, 1/k$. Let $\varepsilon':= 2\varepsilon$ and $n\geq n_0$. First we claim that $F(n,k)\geq 2^{t_{k-1}(n)}$. To see this, first note that any graph $G$ that contains $\overline{T_{k-1}(n)}$ is induced-$C_{2k}$-free (since for any set of $2k$ vertices on $G$, $3$ of them must form a triangle). Moreover, there are precisely $2^{t_{k-1}(n)}$ such graphs for any given labelling of the vertices, which proves the claim.

By Theorem~\ref{thm: containers} (with $C_{2k}, n$ and $\varepsilon'$ taking the roles of $H, N$ and $\varepsilon$ respectively) there is a collection $\mathcal{C}$ of complete $2$-coloured multigraphs on vertex set $[n]$ satisfying properties (a)--(c). In particular, by (a), every induced-$C_{2k}$-free graph on vertex set $[n]$ is contained in some $G\in \mathcal{C}$. Let $\mathcal{C}_1$ be the family of all those $G\in \mathcal{C}$ for which $|G_R\cap G_B|\geq t_{k-1}(n) - \varepsilon' n^2$. Then the number of (labelled) induced-$C_{2k}$-free graphs not contained in some $G\in \mathcal{C}_1$ is at most
\[
|\mathcal{C}| \, 2^{t_{k-1}(n) - \varepsilon' n^2} \leq 2^{- \eps n^2} F(n,k),
\]
because $|\mathcal{C}|\leq 2^{n^{2-\varepsilon'}}$, by (c), and $F(n,k)\geq 2^{t_{k-1}(n)}$. We claim that for every $G\in \mathcal{C}_1$ there exists a complete $2$-coloured multigraph $\tilde{G}$ and a $k$-template $T$ on partition $Q=\{Q_0, Q_1,\dots, Q_{k-2}\}$ such that
$$\tilde{G}_R\cap Q_i^{(2)}=E(T[Q_i]) \hspace{0.6cm} \text{and} \hspace{0.6cm} \tilde{G}_R\cap \tilde{G}_B\cap Q_i^{(2)}=\emptyset$$
for every $i\in \{0,1,\dots, k-2\}$, and $\text{dist}(G, \tilde{G})\leq \eta n^2$. (Note that this claim implies that every induced-$C_{2k}$-free graph contained in $G$ can be made into a $k$-template by changing a total of at most $\eta n^2$ edges within the vertex classes $Q_i$.) Indeed, by (b), each $G\in \mathcal{C}_1$ contains at most $\eps' n^{2k}$ copies of $C_{2k}$. Thus by Lemma~\ref{lem: multigraph removal lemma} there exists a complete $2$-coloured multigraph $G'$ on the same vertex set that is $C_{2k}$-free, such that $\text{dist}(G, G')\leq \delta n^2$. Then $|G'_R\cap G'_B|\geq t_{k-1}(n) - (\varepsilon'+\delta) n^2$. Thus by Lemma~\ref{lem: specialised stability} there exists a complete $2$-coloured multigraph $G''$ on the same vertex set, with $G''_R\cap G''_B = E(T_{k-1}(n))$ and such that $\text{dist}(G', G'')\leq \gamma n^2$. Note that $G''$ can contain at most $\gamma n^{2k}$ copies of $C_{2k}$, since $G'$ is $C_{2k}$-free. Let $Q=\{Q_0, Q_1, \dots, Q_{k-2}\}$ be the unique $(k-1)$-partition of $V(G'')$ such that no partition class induces an edge in $G''_R\cap G''_B$. Thus by Lemma~\ref{lem: specialised removal lemma}, there exists a $k$-template $T=(V(G), E^T)$ on $Q$ such that $|G''_R\triangle E^T|\leq \beta n^2$. Define $\tilde{G}$ to be the $2$-coloured multigraph with $\tilde{G}_R\cap \tilde{G}_B = G''_R\cap G''_B$ and $\tilde{G}_R\cap Q_i=E(T[Q_i])$ for every $i\in \{0,1,\dots, k-2\}$. Then $\text{dist}(G, \tilde{G})\leq (\delta+\gamma+\beta)n^2 \leq \eta n^2$, and $\tilde{G}$ satisfies the required properties. This proves the claim and thus the lemma.
\endproof

\section{The number of $k$-templates}\label{sec: number of templates}

For $k\geq 4$ we denote the set of all $k$-templates on $n$ vertices by $T(n,k)$. Let $T_Q(n,k)$ denote the set of all $k$-templates on $n$ vertices for which $Q$ is an ordered $(k-1)$-partition. The aim of this section is to estimate $|T_Q(n,k)|$ and $|T(n,k)|$ (see Lemmas~\ref{size of template} and~\ref{the number of templates} respectively). Before we start with this we need to introduce some more notation. A \emph{$k$-sun} is defined as follows.
\begin{itemize}
\item If $k=4$, a $k$-sun is any sun (as defined in Section~\ref{subsec: rough 3}).
\item If $k=5$, a $k$-sun is a star or a clique.
\item If $k\geq 6$, a $k$-sun is a star or a triangle.
\end{itemize}
Note that the results of this section are only needed for Theorem~\ref{main theorem} (and not Theorem~\ref{secondary theorem}) and so we would only need to consider the case $k\geq 6$. However, including the cases $k=4,5$ makes little difference to the proofs, and are also interesting in their own right, so we work with all $k\geq 4$ throughout this section.

Let $F_k(n)$ denote the set of all $n$-vertex graphs whose complement is a disjoint union of $k$-suns. Define $f_k(n):= |F_k(n)|$. A pair of vertices $x,y$ is called a {\em twin pair} if $N(x)\backslash \{y\} = N(y)\backslash \{x\}$.

The following two lemmas give some estimates of the value of $f_k(n)$. Note that we do not make use of the upper bound in Lemma~\ref{f-estimate-1} anywhere in this paper, but we include it for its intrinsic interest. It would not be difficult to obtain more accurate bounds, though an asymptotic formula would probably require more work.

\begin{lemma}\label{f-estimate-1}
For all $n\in \mathbb{N}$ and $k\geq 4$,
$$2^{n\log n-en\log\log n} \leq f_k(n) \leq 2^{n\log n - n \log\log n + n}.$$
\end{lemma}
\begin{proof}
Let $P(n)$ denote the number of partitions of an $n$ element set. It is well known (see e.g.~\cite{NGdB}) that
$$2^{n\log n-en\log\log n} \leq P(n)\leq 2^{n\log n-n\log\log n}.$$
We will count the number $f_k(n)$ of graphs $G\in F_k(n)$. Note that $f_k(n)\geq P(n)$ follows by considering each partition class as the vertex set of a star in $\overline{G}$. This then immediately yields the lower bound in Lemma~\ref{f-estimate-1}. Now note that if we choose a partition of $[n]$ into the vertex sets of disjoint suns in $\overline{G}$ (for which there are at most $2^{n\log n-n\log\log n}$ choices), and then for every vertex choose whether the vertex will be in the body of its sun or side of its sun (for which there are a total of $2^n$ choices), we can generate every possible graph $G\in F_k(n)$ (note that some such graphs can be generated by multiple different choices). This yields the upper bound in Lemma~\ref{f-estimate-1}.
\end{proof}

\begin{lemma}\label{f-estimate-2} For $k \geq 4$ and $n>s\geq 10^7$,
$$ s^{s/2} \leq \frac{f_k(n)}{f_k(n-s)} \hspace{0.6cm} \text{and} \hspace{0.6cm} \frac{f_k(n)}{f_k(n-1)} \leq n^2.$$
\end{lemma}
\begin{proof}
By Lemma \ref{f-estimate-1}, $f_k(n) \geq f_k(s) f_k(n-s) \geq 2^{s\log{s}-es\log\log{s}} f_k(n-s) \geq 2^{s\log s/2} f_k(n-s)$, which gives us the lower bound in the statement of the lemma.

For the upper bound, note that every graph in $F_k(n)$ has a twin pair.
For any twin pair $i,j \in [n]$ the number of graphs in $F_k(n)$ for which $i,j$ are twins is at most $2 f_k(n-1)$, since every such graph can be obtained from a graph in $F_k(n-1)$ on vertex set $[n]\setminus \{i\}$ by adding the vertex $i$ and choosing whether to add the edge $ij$ (note that all other edges incident to $i$ are prescribed, since $i,j$ are twins). Thus
$$f_k(n) \leq \sum_{0<i\leq n-1} \sum_{i<j\leq n} 2f_k(n-1) \leq n^2 f_k(n-1),$$
as required.
\end{proof}

The following proposition can be proved by a simple but tedious calculation, which we omit here\COMMENT{Note that (i) is taken directly from~\cite{KOTZ}. The proof of (ii) is as follows.\\
\noindent{\bf Proof of (ii):} Without loss of generality we assume $T_{k-1}(n-s)$ to be such that its vertex classes are subsets of the corresponding vertex classes of $T_{k-1}(n)$. Let $S:=V(T_{k-1}(n))\backslash V(T_{k-1}(n-s))$. We may assume that $T_{k-1}(n)[S]=T_{k-1}(s)$. Then
$$t_{k-1}(n)=e(T_{k-1}(n))=e(T_{k-1}(n-s))+\left(\sum_{v\in S} d_{T_{k-1}(n)}(v)\right) - e(T_{k-1}(n)[S]),$$
and
\begin{itemize}
\item $e(T_{k-1}(n-s))=t_{k-1}(n-s)$,
\item $e(T_{k-1}(n)[S])=t_{k-1}(s)$,
\item For every $v\in S$, $d_{T_{k-1}(n)}(v)\geq (k-2)n/(k-1)-(k-2)$. Indeed, every vertex class of $T_{k-1}(n)$ has size at least $\lfloor n/k-1 \rfloor\geq n/(k-1) -1$, and $v$ is adjacent to all vertices in every vertex class that is not its own (of which there are $k-2$).
\end{itemize}
Putting all this together yields the result.}.

\begin{proposition}\label{omitted proposition}
Let $k, n\in \mathbb{N}$ with $n\geq k\geq 2$ and let $0<s<n$.
\begin{enumerate}[{\rm (i)}]
\item Suppose $G$ is a $k$-partite graph on $n$ vertices in which some vertex class $A$ satisfies $|A-n/k|\geq s$. Then
$$e(G)\leq t_{k}(n)-s\left( \frac{s}{2}-k\right).$$
\item $t_{k-1}(n)\geq t_{k-1}(n-s) + sn(k-2)/(k-1)-s(k-2)-t_{k-1}(s)$.
\end{enumerate}
\end{proposition}

\begin{lemma}\label{size of template}
Let $k\geq 4$. There exists $n_0\in \mathbb{N}$ such that for every $n\geq n_0$ and every ordered $(k-1)$-partition $Q$ of $[n]$, the number of $k$-templates on $Q$ satisfies
$$|T_Q(n,k)|\leq 2^{6 (\log n)^2}2^{t_{k-1}(n)}f_k\left(n_k \right),$$
where we recall that $n_k:= \left\lceil n/(k-1) \right\rceil$.
\end{lemma}
\begin{proof}
Denote the classes of $Q$ by $Q_0,Q_1,\dots, Q_{k-2}$ and let $b:= ||Q_0| - \lceil \frac{n}{k-1} \rceil|$. Then by Proposition~\ref{omitted proposition}(i) the number of $k$-templates on this partition is at most

$$f_k(|Q_0|) 2^{\sum_{0\leq i<j\leq k-2} |Q_i||Q_j|} \leq f_k\left(n_k +b\right) 2^{t_{k-1}(n) - b\left(b/2-(k-1)\right)}.$$
Let $h(b):=f_k(n_k +b) 2^{t_{k-1}(n) - b(b/2-(k-1))}$. Then by Lemma \ref{f-estimate-2},
$$\frac{h(b+1)}{h(b)} \leq \left(\frac{n}{k-1}+b+2\right)^2 2^{-\left((2b+1)/2-(k-1)\right)}. $$
Thus $h(b)$ is a decreasing function for $b\geq 3\log n$. This together with Lemma \ref{f-estimate-2} gives us that the number of $k$-templates on $Q$ is at most
\begin{align*}
h(b) &\leq f_k\left(n_k + 3\log n\right) 2^{t_{k-1}(n)} \leq (n^2)^{3\log n} 2^{t_{k-1}(n)} f_k\left(n_k \right)\\
&= 2^{6 (\log n)^2}2^{t_{k-1}(n)}f_k\left(n_k \right),
\end{align*}
as required.
\end{proof}

We call a component of a graph \emph{non-trivial} if it contains at least $2$ vertices. The proof of Lemma~\ref{the number of templates} will make use of the following proposition.

\begin{proposition}\label{prop: two components}
Let $k\geq 4$. There exists $n_0\in \mathbb{N}$ such that the following holds for every $n\geq n_0$. Let $Q$ be a balanced ordered $(k-1)$-partition of $[n]$. The proportion of $k$-templates $G$ on $Q$ that are such that $\overline{G}[Q_0]$ has at most one non-trivial component is at most $2^{-n}$.
\end{proposition}
\begin{proof}
Since $Q$ is balanced, the number of $k$-templates on $Q$ is at least $2^{t_{k-1}(n)}f_k(\lfloor \frac{n}{k-1}\rfloor )$.

We can generate all possible edge sets for $G[Q_0]$ such that $\overline{G}[Q_0]$ has at most one non-trivial component in the following way. Note that for every such $G[Q_0]$, $\overline{G}[Q_0]$ contains at most one disjoint sun $S$ of order at least two. For every vertex in $Q_0$ we choose whether it will belong to the body of $S$, the side of $S$, or neither (for which there are a total of at most $3^n$ choices). Hence the number of $k$-templates $G$ on $Q$ that are such that $\overline{G}[Q_0]$ has at most one non-trivial component is at most $3^{n} 2^{t_{k-1}(n)}$.

Since we have by Lemma~\ref{f-estimate-1} that $f_k(m)\geq 2^{m\log m-em\log\log m}$ for all $m\in \mathbb{N}$, the result follows (with some room to spare).
\end{proof}

The following trivial observation will be useful in the proof of Lemma~\ref{the number of templates}.
\begin{equation}\label{induced four cycle observation}
\textrm{If a graph $G$ is a disjoint union of suns then $G$ contains no induced $4$-cycles.}
\end{equation}

\begin{lemma}\label{the number of templates}
For every $k\geq 4$ there exists $n_0\in \mathbb{N}$ such that the following holds for all $n\geq n_0$, where we recall that $n_k:= \left\lceil n/(k-1) \right\rceil$. The number of $k$-templates on vertex set $[n]$ satisfies
$$|T(n,k)| \geq \frac{(k-1)^n}{2(k-2)! n^k}2^{t_{k-1}(n)}f_k\left(n_k \right).$$

\end{lemma}
\begin{proof}
Choose $n_0$ such that $1/n_0\ll 1/k$, and let $n\geq n_0$. Given a $k$-template $G$ on vertex set $[n]$ and an ordered $(k-1)$-partition $Q=(Q_0,\{Q_1,\dots,Q_{k-2}\})$ of $[n]$, we say that $G$ is $Q$-compatible if $G$ is a $k$-template on $Q$ and the following hold:

\begin{enumerate}
\item[$(\alpha)$] Whenever $\ell\leq 2k$ and $0\leq i\leq k-2$ and $v_1,v_2,\dots, v_{\ell} \in V(G)\setminus Q_i$, we have that
$$|\overline{N}_{Q_i}(\{v_1,v_2,\dots,v_{\ell}\})| \geq \frac{n}{2^{\ell+1}(k-1)}.$$
\item[$(\beta)$] $\overline{G}[Q_0]$ has at least $2$ non-trivial components.
\end{enumerate}

\vspace{0.3cm}
\noindent {\bf Claim 1:} \emph{Given a balanced ordered $(k-1)$-partition $Q=(Q_0,\{Q_1,\dots,Q_{k-2}\})$ of $[n]$, the number of $k$-templates $G$ on vertex set $[n]$ that are $Q$-compatible is at least $2^{t_{k-1}(n)-1}f_k(n_k)/n^2$.}

\noindent Indeed, consider a random graph $G$ where for each potential crossing edge with respect to $Q$ we choose the edge to be present or not, each with probability $1/2$, independently; we let $G[Q_0]$ be one of the $f_k(|Q_0|)$ graphs in $F_k(|Q_0|)$, chosen uniformly at random; and we choose all edges to be present inside $Q_i$ for every $i>0$. So each $k$-template on $Q$ is equally likely to be generated. Note that the number of potential crossing edges with respect to $Q$ is $2^{t_{k-1}(n)}$. This together with Lemma~\ref{f-estimate-2} implies that the number of graphs in the probability space is at least $2^{t_{k-1}(n)}f_k(n_k)/n^2$. By Lemma~\ref{Chernoff Bounds}(ii) and Proposition~\ref{prop: two components} respectively, we have that at least half of all graphs $G$ in the probability space satisfy $(\alpha)$ and $(\beta)$, which proves the claim.
\vspace{0.3cm}


\vspace{0.3cm}
\noindent {\bf Claim 2:} \emph{Given two balanced ordered $(k-1)$-partitions $Q=(Q_0,\{Q_1,\dots,Q_{k-2}\})$ and $Q'=(Q_0',\{Q_1',\dots,Q'_{k-2}\})$ of $[n]$, and a $k$-template $G$ on $[n]$ that is both $Q$-compatible and $Q'$-compatible, there exist $k$ vertices $u_0,v_0,v_1,\dots, v_{k-2}\in [n]$ that are such that $G[\{u_0,v_0,v_1,\dots, v_{k-2}\}]$ contains exactly one edge $u_0v_0$ and $u_0\in Q_0\cap Q'_0$ and $v_i\in Q_i\cap Q_i'$ for all $i\geq 0$.}

\noindent To show this, we first choose a set of $2k$ vertices $U=\{u_{0,1}, w_{0,1}, u_{0,2}, w_{0,2}, u_1, w_1, \dots, u_{k-2}, w_{k-2}\}$ such that $u_{0,1}, w_{0,1}, u_{0,2}, w_{0,2}\in Q_0$ and $u_i, w_i\in Q_i$ for every $i>0$ and
$$E(G[U])=\{u_{0,1}u_{0,2}, u_{0,2}w_{0,1}, w_{0,1}w_{0,2}, w_{0,2}u_{0,1}, u_1 w_1, \dots, u_{k-2} w_{k-2}\}.$$
This is possible since $G$ satisfies $(\alpha), (\beta)$ with respect to $Q$. Now if there exist distinct $i,j>0$ such that $u_i, w_i, u_j, w_j \in Q_0'$ then $\overline{G}[Q_0']$ contains the induced $4$-cycle $u_i u_j w_i w_j$, which by (\ref{induced four cycle observation}) contradicts the fact that $G$ is a $k$-template on $Q'$. So, by relabelling vertices if necessary, we may assume that $u_2,\dots, u_{k-2}\notin Q_0'$. If $u_{0,1},w_{0,1}\notin Q_0'$ then by the pigeon-hole principle there must exist $i>0$ such that $Q_i'$ contains at least $2$ elements of $\{u_{0,1},w_{0,1}, u_2,\dots, u_{k-2} \}$, contradicting the assumption that $G[Q_i']$ is a clique. So, by relabelling vertices if necessary, we may assume that $u_{0,1}\in Q_0'$, and similarly that $u_{0,2}\in Q_0'$. Now if $u_1, w_1 \in Q_0'$ then $\overline{G}[Q_0']$ contains the induced $4$-cycle $u_{0,1} u_1 u_{0,2} w_1$, which by (\ref{induced four cycle observation}) contradicts the fact that $G$ is a $k$-template on $Q'$. So, by relabelling vertices if necessary, we may assume that $u_1\notin Q_0'$, and thus $u_1,\dots, u_{k-2}\notin Q_0'$. Recall that for all $i>0$, $G[Q_i']$ is a clique, so $Q_i'$ can contain at most one vertex in $\{u_1,\dots, u_{k-2}\}$. Thus we may assume, by relabelling indices if necessary, that $u_{0,1}, u_{0,2}\in Q_0\cap Q_0'$ and $u_i\in Q_i\cap Q_i'$ for every $i>0$. So setting $u_0:= u_{0,1}, v_0:= u_{0,2}$ and $v_i:= u_i$ for all $i>0$ yields the required set of vertices.
\vspace{0.3cm}

\vspace{0.3cm}
\noindent {\bf Claim 3:} \emph{If there exist balanced ordered $(k-1)$-partitions $Q=(Q_0,\{Q_1,\dots,Q_{k-2}\})$ and $Q'=(Q_0',\{Q_1',\dots,Q'_{k-2}\})$ of $[n]$, and a $k$-template $G$ on $[n]$ that is both $Q$-compatible and $Q'$-compatible, then $Q=Q'$.}

\noindent Consider any $k$ vertices $u_0, v_0,\dots, v_{k-2}\in V(G)$ that are such that
$G[\{u_0,v_0,v_1,\dots, v_{k-2}\}]$ contains exactly one edge $u_0v_0$ and $u_0\in Q_0\cap Q'_0$ and $v_i\in Q_i\cap Q_i'$ for all $i\geq 0$. Such vertices exist by Claim~2.
For $i>0$ define $$\overline{N}_i  := \overline{N}_{Q_i}( \{u_0,v_0,\dots, v_{k-2}\} \setminus \{v_i\}).$$
$$\overline{N}'_i  := \overline{N}_{Q'_i}( \{u_0,v_0,\dots, v_{k-2}\} \setminus \{v_i\}).$$
Since both $\overline{N}_i$ and $\overline{N}'_i$ are subsets of the common non-neighbourhood of $\{u_0,v_0,v_1\dots, v_{k-2}\} \setminus \{v_i\}$, neither can intersect $Q_j$ or $Q'_j$ for $j\notin \{0,i\}$. Note that all vertices in $\overline{N}_i$ are adjacent.
Thus $|\overline{N}_i \cap Q'_0| \leq 1$, since otherwise $\overline{G}[Q'_0]$ contains an induced $4$-cycle on $u_0,v_0$ together with $2$ vertices from $\overline{N}_i$, which by (\ref{induced four cycle observation}) contradicts the fact that $G$ is a $k$-template on $Q'$. Similarly, $|\overline{N}'_i \cap Q_0| \leq 1$. Define
$$\overline{N}^{\dagger}_i  :=(\overline{N}_i \cup \overline{N}'_i) \setminus (Q_0\cup Q'_0).$$
Then $\overline{N}^{\dagger}_i \subseteq Q_i\cap Q_i'$.

Now we consider any vertex $w\in Q_0$. Since $G$ satisfies $(\alpha)$ with respect to $Q$, we have that for every $i> 0$,
\begin{align}\label{eqn: counting templates}
|\overline{N}_{Q'_i}(w)|&\geq |\overline{N}(w)\cap\overline{N}^{\dagger}_i | \geq |\overline{N}(w)\cap\overline{N}_i|-1 \\
&= |\overline{N}_{Q_i}( \{u_0,v_0,\dots, v_{k-2},w\}\setminus \{v_i\})| -1 \geq \frac{n}{2^{k+1}(k-1)} -1 \geq 1.\nonumber
\end{align}
Thus $w$ must belong to $Q'_0$, since $G[Q_i']$ is a clique for every $i> 0$. Hence $Q_0\subseteq Q'_0$. In the same way we can show that $Q'_0\subseteq Q_0$. Thus $Q_0=Q'_0$.

Now we consider any vertex $w\in Q_j$, for $j> 0$. Since $G$ satisfies $(\alpha)$ with respect to $Q$, we have (similarly to (\ref{eqn: counting templates})) that for every $i\neq j$ with $i> 0$,
$$|\overline{N}_{Q'_i}(w)|\geq |\overline{N}(w)\cap \overline{N}^{\dagger}_i | \geq 1.$$
Thus $w\in Q_0'\cup Q_j'$. Together with the fact that $Q_0= Q'_0$ this implies that $w\in Q'_j$. Thus $Q_j\subseteq Q'_j$ for all $j>0$.

Hence $Q=Q'$, which proves the claim.
\vspace{0.3cm}

We now count the number of balanced ordered $(k-1)$-partitions. Since the vertex classes of a balanced ordered $(k-1)$-partition of $[n]$ have sizes $\lceil \frac{n}{k-1} \rceil,\lceil \frac{n-1}{k-1} \rceil, \dots, \lceil \frac{n-k+2}{k-1} \rceil$, the number of such $(k-1)$-partitions is
\begin{equation*}
\frac{1}{(k-2)!}{\binom{n}{\lceil \frac{n}{k-1} \rceil,\lceil \frac{n-1}{k-1} \rceil, \dots, \lceil \frac{n-k+2}{k-1} \rceil}}.
\end{equation*}
This together with Claims~1 and~3 implies that
\begin{equation}\label{eqn: counting templates 2}
|T(n,k)| \geq \frac{1}{2(k-2)!n^2}{\binom{n}{\lceil \frac{n}{k-1} \rceil,\lceil \frac{n-1}{k-1} \rceil, \dots, \lceil \frac{n-k+2}{k-1} \rceil}}2^{t_{k-1}(n)}f_k(n_k).
\end{equation}

Now note that if $a_1+\dots + a_{k-1} = n$, then ${\binom{n}{a_1,a_2,\dots, a_{k-1}}}$ is maximized by taking $a_j := \lceil \frac{n-j+1}{k-1}\rceil $ for every $j$.  This implies that 
$$(k-1)^n = \sum_{a_1+\dots+a_{k-1}=n}{\binom{n}{a_1,a_2,\dots, a_{k-1}} } \leq n^{k-2} {\binom{n}{\lceil \frac{n}{k-1} \rceil,\lceil \frac{n-1}{k-1} \rceil, \dots, \lceil \frac{n-k+2}{k-1} \rceil}},$$
which together with (\ref{eqn: counting templates 2}) implies the result.
\end{proof}

\section{Properties of near-$k$-templates}\label{sec: set-up}

In this section we collect some properties of graphs which are close to being $k$-templates. In particular, when $k\geq 6$, this means we consider graphs $G$ which have a vertex partition such that each vertex class induces on $G$ an almost complete graph. (As in the previous section, we will need the results of this section for our main results only for the case $k\geq 6$, but we prove the results for all $k\geq 4$ since it makes little difference to the proofs.) More formally, given $k\geq 4$, a graph $G$ on vertex set $[n]$, and an ordered $(k-1)$-partition $Q$ of $[n]$ we define
$$h(Q,G):=\sum_{i=0}^{k-2} |E(\overline{G}[Q_i])|.$$
We say $Q$ is an {\em optimal ordered $(k-1)$-partition of $G$} if $h(Q,G)$ is the minimum value $h(Q',G)$ takes over all partitions $Q'$ of $[n]$. Note that if $h(Q,G)=0$ then $G$ is a $k$-template on $Q$, and that the following also holds.
\begin{equation}\label{eqn: near templates}
\textrm{If $k\geq 6$ then every $k$-template $G'$ on $Q$ satisfies $h(Q,G')\leq n$.}
\end{equation}
Note that (\ref{eqn: near templates}) does not hold for $k\in \{4,5\}$. We will require the following definitions in what follows.

\begin{itemize}
\item Recall that $F(n,k)$ denotes the set of all labelled induced-$C_{2k}$-free graphs on vertex set $[n]$.
\item Given $n\in \mathbb{N}$, $k\geq 4$, and $\eta>0$, we define $F(n,k,\eta)\subseteq F(n,k)$ to be the set of all graphs in $F(n,k)$ such that $h(Q,G)\leq \eta n^2$ for some optimal ordered $(k-1)$-partition $Q$ of $G$.
\item Given further an ordered $(k-1)$-partition $Q=(Q_0,\{Q_1,\dots, Q_{k-2}\})$ of $[n]$ we define $F_Q(n,k)\subseteq F(n,k)$ to be the set of all graphs in $F(n,k)$ for which $Q$ is an optimal ordered $(k-1)$-partition
\item Similarly we define $F_Q(n,k,\eta)\subseteq F(n,k,\eta)$ to be the set of all graphs in $F(n,k,\eta)$ for which $Q$ is an optimal ordered $(k-1)$-partition.
\end{itemize}

Recall that, given a graph $G$ on vertex set $[n]$ and an index $i\in \{0,1,\dots, k-2\}$, we let $d^i_{G,Q}(x), \overline{d}^i_{G,Q}(x)$ denote the number of neighbours and non-neighbours of $x$ in $Q_i$, respectively. The following proposition follows immediately\COMMENT{{\bf Proof of Proposition~\ref{optimality}:} Suppose for a contradiction that there exists a vertex $x \in Q_i$ that satisfies $\overline{d}^j_{G,Q}(x) < \overline{d}^i_{G,Q}(x)$. Define an ordered $(k-1)$-partition $Q'=(Q'_0,\{Q'_1,\dots, Q'_{k-2}\})$ by $Q'_j:= Q_j\cup \{v\}, Q'_i:= Q_i\setminus \{x\}$ and $Q'_\ell:=Q_\ell$ for all $\ell\in \{0,1,\dots, k-2\}\backslash\{i,j\}$. Then $h(Q',G) = h(Q,G)+\overline{d}^j_{G,Q}(x)-\overline{d}^i_{G,Q}(x)< h(Q,G)$, which contradicts the optimality of $Q$ and hence completes the proof.} from the definition of optimality.

\begin{proposition}\label{optimality}
Let $k\geq 4$, let $\eta>0$, let $Q=(Q_0,\{Q_1,\dots, Q_{k-2}\})$ be an ordered $(k-1)$-partition of $[n]$, and let $G\in F_Q(n,k,\eta)$. For any two distinct indices $i,j\in \{0,1,\dots, k-2\}$ every vertex $x\in Q_i$ satisfies $\overline{d}^j_{G,Q}(x) \geq \overline{d}^i_{G,Q}(x)$.
\end{proposition}

Next we show that for most graphs which are close to being $k$-templates, the bipartite graphs between the partition classes are quasirandom.

Given $k\geq 4$, $\nu=\nu(n)>0$ and an ordered $(k-1)$-partition $Q=(Q_0,\{Q_1,\dots, Q_{k-2}\})$ of $[n]$, we define the following properties that a graph on vertex set $[n]$ may satisfy with respect to $Q$.

\begin{enumerate}
\item[$({\rm F}1)_{\nu}$] If $U_i\subseteq Q_i$ and $U_j\subseteq Q_j$ with $|U_i||U_j|\geq \nu^2 n^2$ for distinct $0\leq i,j\leq k-2$, then $ \frac{1}{4} \leq \frac{|e(U_i,U_j)|}{|U_i||U_j|} \leq \frac{3}{4}$.
\item[$({\rm F}2)_{\nu}$] $||Q_i|-\frac{n}{k-1}|\leq \nu n$ for every $0\leq i\leq k-2$.
\end{enumerate}

Given $\eta, \mu>0$ we define $F_Q(n,k,\eta,\mu)$ to be the set of all graphs in $F_Q(n,k,\eta)$ that satisfy $({\rm F}1)_{\mu}$ and $({\rm F}2)_{\mu}$ with respect to $Q$.

\begin{lemma}\label{pre regular lemma}
Let $n\geq k\geq 4$, let $0<\eta<1$, let $6k/n\leq \nu=\nu(n)\leq 1$, let $6\log n\leq m=m(n)\leq 10^{-11}n^2$, and let $Q=(Q_0,\{Q_1,\dots, Q_{k-2}\})$ be an ordered $(k-1)$-partition of $[n]$. Then the following hold.
\begin{enumerate}[{\rm (i)}]
\item The number of graphs $G$ in $F_Q(n,k,\eta)$ that fail to satisfy $({\rm F}1)_{\nu}$ with respect to $Q$ and that have at most $m$ internal non-edges is at most $2^{t_{k-1}(n) + \xi(m/n^2)n^2} 2^{2n+1} \exp(-\nu^2n^2/32)$.
\item The number of graphs $G$ in $F_Q(n,k,\eta)$ that fail to satisfy $({\rm F}2)_{\nu}$ with respect to $Q$ and that have at most $m$ internal non-edges is at most $2^{t_{k-1}(n) + \xi(m/n^2)n^2} \exp( -\nu^2 n^2/6 )$.
\end{enumerate}
\end{lemma}
\begin{proof}
For both (i) and (ii) we consider constructing such a graph $G$. By (\ref{entropy bound}) there are at most $\binom{n^2}{\leq m} \leq 2^{\xi(m/n^2) n^2}$ choices for the internal edges of $G$.

We first prove (i). For a given choice of internal edges, consider the random graph $H$ where for each possible crossing edge with respect to $Q$ we choose the edge to be present or not, with probability $1/2$, independently. Note that the total number of ways to choose the crossing edges is at most $2^{t_{k-1}(n)}$, and each possible configuration of crossing edges is equally likely. So an upper bound on the number of graphs $G \in F_Q(n,k,\eta)$ that fail to satisfy property $({\rm F}1)_{\nu}$ with respect to $Q$ and that have at most $m$ internal non-edges is
\begin{equation}\label{F intermediate bound}
2^{t_{k-1}(n) + \xi(m/n^2)n^2 }\mathbb{P}(H \text{ fails to satisfy } ({\rm F}1)_{\nu} \text{ with respect to } Q).
\end{equation}
Note that the number of choices for $U_i\subseteq Q_i, U_j\subseteq Q_j$ with $|U_i||U_j|\geq \nu^2 n^2$ is at most $2^{2n}$ and that $\mathbb{E}(e(U_i,U_j)) = |U_i||U_j|/2 \geq \nu^2 n^2/2$. Hence by Lemma~\ref{Chernoff Bounds},
\begin{equation*}
\mathbb{P}(H \text{ fails to satisfy } ({\rm F}1)_{\nu} \text{ with respect to } Q) \leq 2^{2n+1} \exp\left(-\frac{\nu^2 n^2}{32}\right).
\end{equation*}
This together with (\ref{F intermediate bound}) yields the result.

We now prove (ii). If $||Q_i|-\frac{n}{k-1}|>\nu n$ for some $0\leq i\leq k-2$, then by Proposition~\ref{omitted proposition}(i) the number of crossing edges in $G$ is at most
$$ t_{k-1}(n) -  \frac{\nu^2n^2}{3}.$$
We can conclude that the number of $G \in F_Q(n,k,\eta)$ that fail to satisfy $({\rm F}2)_{\nu}$ with respect to $Q$ and that have at most $m$ internal non-edges is at most
\begin{equation*}
2^{\xi(m/n^2)n^2} 2^{t_{k-1}(n) -\frac{\nu^2n^2}{3} }\leq 2^{t_{k-1}(n) + \xi(m/n^2)n^2}\exp\left( -\frac{\nu^2 n^2}{6} \right),
\end{equation*}
as required.
\end{proof}

We will apply the following special case of Lemma~\ref{pre regular lemma} in Section~\ref{sec: final calculation} in the proof of Lemma~\ref{induction conclusion}.

\begin{corollary}\label{regular lemma}
Let $k\geq 4$ and let $0<\eta,\mu<10^{-11}$ be such that $\mu^2 > 24\xi(\eta)$. There exists an integer $n_0 = n_0(\mu,k)$ such that for all $n\geq n_0$ and every ordered $(k-1)$-partition $Q$ of $[n]$,
$$ |F_Q(n,k,\eta)\setminus F_Q(n,k,\eta,\mu)| \leq 2^{t_{k-1}(n) - \frac{\mu^2n^2}{100}}.$$
\end{corollary}
\begin{proof}
We choose $n_0$ such that $1/n_0 \ll \eta, \mu, 1/k$. Applying Lemma~\ref{pre regular lemma} with $\mu, \eta n^2$ playing the roles of $\nu, m$ respectively yields that
$$ |F_Q(n,k,\eta)\setminus F_Q(n,k,\eta,\mu)| \leq 2^{t_{k-1}(n)+\xi(\eta)n^2}2^{2n+1} \left( e^{- \frac{\mu^2 n^2}{6} }+ e^{- \frac{\mu^2 n^2}{32} } \right) \leq 2^{t_{k-1}(n) - \frac{\mu^2 n^2}{100}},$$
as required.
\end{proof}

The next proposition follows immediately from~\cite[Lemma~2.22]{BaBu}. We will use it to find induced copies of $C_{2k}$. (Usually $T$ will be a suitable induced subgraph of $C_{2k}$ and the $A_i,B_i$ will be the intersection of (non-)neighbourhoods of vertices that we have already embedded.)

\begin{proposition} \label{building}
Let $n_0,k \in \mathbb{N}$ and $\eta,\mu >0$ be chosen such that $k\geq 4$ and $1/n_0\ll \eta \ll \mu \ll 1/k$. Then the following holds for all $n\in \mathbb{N}$ with $n\geq n_0$. Let $Q=(Q_0,\{Q_1,\dots, Q_{k-2}\})$ be an ordered $(k-1)$-partition of $[n]$ and suppose $G\in F_Q(n,k,\eta,\mu)$. Let $I\subseteq \{0,1,\dots, k-2\}$. For every $i\in I$ let $A_i,B_i\subseteq Q_i$ be disjoint with $|A_i|, |B_i| \geq \mu^{1/2} n$. Let $T$ be a $2|I|$-vertex graph with a perfect matching whose edges are $v_iu_i$ for every $i \in I$. Then there exists an injection $f:V(T) \rightarrow V(G)$ such that $f(v_i) \in A_i, f(u_i) \in B_i$ for every $i\in I$, and $f(V(T))$ induces on $G$ a copy of $T$.
\end{proposition}

Finally we show that if $G$ is close to being a $k$-template then removing a small number of vertices from $G$ does not alter its optimal ordered $(k-1)$-partition very much.

Given $m,n\in \mathbb{N}$ and an ordered $(k-1)$-partition $Q$ of $[n]$, we define $\mathcal{P}(Q,m)$ to be the collection of all ordered $(k-1)$-partitions of $[n]$ that can be obtained from $Q$ by moving at most $m$ vertices between partition classes, and possibly choosing a different partition class to be the labelled one. Then it is easy to see that
\begin{equation}\label{size mathcal{P}}
|\mathcal{P}(Q,m)| \leq k {\binom{n}{m}} k^m \leq k\left(\frac{ekn}{m}\right)^m  \leq 2^{m\log (ek^2 n/m)}.
\end{equation}

Given an ordered $(k-1)$-partition $Q$ of $[n]$ and a set $S\subseteq [n]$, let $Q-S$ denote the ordered $(k-1)$-partition (possibly with some empty classes) obtained from $Q$ by deleting all elements of $S$ from their partition classes.

\begin{lemma} \label{pre new partition}
Let $k\geq 4$, let $0<\eta,\mu\leq 1/k^3$, let $0<\nu=\nu(n)\leq 1/k^3$, and let $0\leq m=m(n)\leq n^2$ with $\nu^2 > 4m/n^2$ for all $n\in \mathbb{N}$. There exists $n_0\in \mathbb{N}$ such that the following holds for all $n\geq n_0$. Let $Q=(Q_0,\{Q_1,\dots, Q_{k-2}\})$ be an ordered $(k-1)$-partition of $[n]$ and let $S\subseteq [n]$ with $|S|\leq n/k^2$. Then for every $G\in F_Q(n,k,\eta,\mu)$ that satisfies $({\rm F}1)_{\nu}$ with respect to $Q$ and that has at most $m$ internal non-edges, every optimal ordered $(k-1)$-partition of $G - S$ is an element of $\mathcal{P}(Q-S,k^4\nu^2 n)$.
\end{lemma}
\begin{proof}
Let $G\in F_Q(n,k,\eta, \mu)$ have at most $m$ internal non-edges and satisfy $({\rm F}1)_{\nu}$ with respect to $Q$, and let $Q'=(Q_0',\{Q_1',\dots, Q_{k-2}'\})$ be an optimal ordered $(k-1)$-partition of $G-S$. By optimality of $Q'$ it must be that $G-S$ has at most $m$ internal non-edges with respect to $Q'$.

For every $i \in \{0,1,\dots,k-2\}$, since $|Q_i - S| \geq n/(k-1) - \mu n - n/k^2 \geq n/k$, the pigeon-hole principle implies that there exists $j \in \{0,1,\dots, k-2\}$ such that $|Q_i\cap Q_{j}'|\geq n/k^2$. We define a function $\sigma$ by setting $\sigma(i)$ to be an index in $\{0,1,\dots,k-2\}$ that satisfies $|Q_i\cap Q_{\sigma(i)}'|\geq n/k^2$, for every $i \in \{0,1,\dots,k-2\}$. Suppose for a contradiction that there exists $i'\in \{0,1,\dots, k-2\}$ with $i\ne i'$ such that $|Q_{i'}\cap Q_{\sigma(i)}'| \geq k^2\nu^2 n$. Then since $G$ satisfies $({\rm F}1)_{\nu}$ with respect to $Q$ we have that the number of internal non-edges in $G-S$ with respect to $Q'$ is at least $|Q_{i}\cap Q_{\sigma(i)}'||Q_{i'}\cap Q_{\sigma(i)}'|/4 \geq \nu^2 n^2/4>m$. This contradicts our previous observation that $G-S$ has at most $m$ internal non-edges with respect to $Q'$. Hence $\sigma$ is a permutation on $\{0,1,\dots,k-2\}$. Moreover $|Q_i \cap Q'_{j}| < k^2\nu^2 n$ for all $j\in \{0,1,\dots, k-2\}$ with $j\neq \sigma(i)$.

Let $\mathcal{P}$ be the set of all ordered $(k-1)$-partitions of $[n]\backslash S$ for which such a permutation exists. So by the above we have that for every $G\in F_Q(n,k,\eta)$ that satisfies $({\rm F}1)_{\nu}$ with respect to $Q$ and that has at most $m$ internal non-edges, every optimal ordered $(k-1)$-partition of $G-S$ is an element of $\mathcal{P}$. So it remains to show that $\mathcal{P}\subseteq \mathcal{P}(Q-S,k^4\nu^2 n)$. This follows from the observation that every element of $\mathcal{P}$ can be obtained by starting with the (labelled) $(k-1)$-partition $Q_0\setminus S, Q_1 \setminus S,\dots, Q_{k-2} \setminus S$, applying a permutation of $\{0,1,\dots,k-2\}$ to the partition class labels, then for every ordered pair of partition classes moving at most $k^2\nu^2 n$ elements from the first partition class to the second, and finally unlabelling all but one of the resulting partition classes.
\end{proof}

The following is an immediate corollary of Lemma~\ref{pre new partition}, applied with $\mu, \eta n^2$ playing the roles of $\nu, m$, respectively.

\begin{corollary}\label{new partition}
Let $k\geq 4$ and $0<\eta,\mu< 1/k^3$ with $\mu^2 > 4\eta$. There exists $n_0\in \mathbb{N}$ such that the following holds for all $n\geq n_0$. Let $Q=(Q_0,\{Q_1,\dots, Q_{k-2}\})$ be an ordered $(k-1)$-partition of $[n]$ and let $S\subseteq [n]$ with $|S|\leq n/k^2$. Then for every $G\in F_Q(n,k,\eta,\mu)$, every optimal ordered $(k-1)$-partition of $G - S$ is an element of $\mathcal{P}(Q-S,k^4\mu^2 n)$.
\end{corollary}

\section{Derivation of Theorem~\ref{main theorem} from the main lemma}\label{sec: main proof}

The following lemma is the key result in our proof of Theorem~\ref{main theorem}. Together with Lemma~\ref{size of template} it implies that, for $k\geq 6$, almost all induced-$C_{2k}$-free graphs $G$ with a given optimal ordered $(k-1)$-partition are $k$-templates\COMMENT{Note that I do mean Lemma~\ref{size of template} and not Lemma~\ref{the number of templates} here, since the sentence is about induced-$C_{2k}$-free graphs \emph{with a given optimal partition}}. Recall that $n_k:=\left\lceil n/(k-1) \right\rceil$, that $f_k(n)$ and $T_Q(n,k)$ were defined at the beginning of Section~\ref{sec: number of templates}, and that $F_Q(n,k)$ was defined at the beginning of Section~\ref{sec: set-up}.

\begin{lemma}\label{induction conclusion}
For every $n,k\in \mathbb{N}$ with $k\geq 6$ there exists $C\in \mathbb{N}$ such that the following holds. For every ordered $(k-1)$-partition $Q$ of $[n]$,
\begin{align*}
|F_Q(n,k)| &\leq |T_Q(n,k)|+ 5C 2^{-n^{\frac{1}{2k^2}}/3}f_k(n_k) 2^{t_{k-1}(n)}.
\end{align*}
\end{lemma}

Lemma~\ref{induction conclusion} will be proved in the remaining sections of this paper. We will now use it to derive Theorem~\ref{main theorem}.

\removelastskip\penalty55\medskip\noindent{\bf Proof of Theorem~\ref{main theorem}.}
Let $n_0\in \mathbb{N}$ be as in Lemma~\ref{the number of templates}, let $C\in \mathbb{N}$ be as in Lemma~\ref{induction conclusion}, let $n_1\in \mathbb{N}$ satisfy $1/n_1 \ll 1/k$, let $n\in \mathbb{N}$ with $n\geq \max \{n_0, n_1\}$, and let $\mathcal{Q}$ be the set of all ordered $(k-1)$-partitions of $[n]$. Since $T(n,k)\subseteq F(n,k)$ and $T_Q(n,k)\subseteq F_Q(n,k)$ for every $Q\in \mathcal{Q}$, Lemma~\ref{induction conclusion} implies that
\begin{align*}
|F(n,k)|-|T(n,k)| &= |F(n,k)\backslash T(n,k)| \leq  \sum_{Q\in \mathcal{Q}} |F_Q(n,k) \setminus T_Q(n,k)|\\
&= \sum_{Q\in \mathcal{Q}} \left( |F_Q(n,k)| - |T_Q(n,k)| \right)\leq 5C (k-1)^n 2^{-n^{\frac{1}{2k^2}}/3} f_k(n_k) 2^{t_{k-1}(n)}\\
&\leq C 2^{-n^{\frac{1}{2k^2}}/4} \frac{(k-1)^n}{2(k-2)! n^k} f_k(n_k) 2^{t_{k-1}(n)}.
\end{align*}
This together with Lemma~\ref{the number of templates} implies that
$$|F(n,k)|-|T(n,k)| \leq C 2^{-n^{\frac{1}{2k^2}}/4} |T(n,k)|= o(|T(n,k)|),$$
where we use the little $o$ notation with respect to $n$. So $|F(n,k)| = (1+o(1))|T(n,k)|$, as required.
\endproof

Sections~\ref{sec: estimate 1}--\ref{sec: final calculation} are devoted to proving Lemma~\ref{induction conclusion} by an inductive argument. For the remainder of the paper we fix constants $C, k, n_0\in \mathbb{N}$ with $k\geq 6$ and $\varepsilon, \eta, \mu, \gamma, \beta, \alpha>0$ such that
\begin{equation}\label{eqn: hierarchy}
\frac{1}{C} \ll  \frac{1}{n_0}\ll \varepsilon \ll \eta \ll \mu \ll \gamma \ll \beta \ll \alpha \ll \frac{1}{k}.
\end{equation}
We also set $M:= R_{2k-2}(\lceil \frac{1}{\gamma}\rceil) +1$, fix an arbitrary integer $n\geq n_0$, and fix an arbitrary ordered $(k-1)$-partition $Q=(Q_0, \{Q_1,\dots, Q_{k-2}\})$ of $[n]$.

We make the following inductive assumption in Sections~\ref{sec: estimate 1}, \ref{sec: estimate 2} and~\ref{sec: estimate 3}: for every $n'\leq n-1$, and every ordered $(k-1)$-partition $Q'=(Q'_0, \{Q'_1,\dots, Q'_{k-2}\})$ of $[n']$,
$$|F_{Q'}(n',k)\setminus T_Q(n',k)| \leq 5 C 2^{-(n')^{\frac{1}{2k^2}}/3} f_k\left(n'_k \right) 2^{t_{k-1}(n')}.$$
Note that this together with Lemma~\ref{size of template} implies that
\begin{equation} \label{induct}
|F_{Q'}(n',k)|\leq 6C 2^{6(\log n')^2} f_k\left( n'_k \right) 2^{t_{k-1}(n')}.
\end{equation}

We now give a number of definitions that will be used in the remaining sections. Given an index $i\in \{0,1,\dots, k-2\}$, we call a vertex $x$ of a graph $G$ {\em $i$-light} if at least one of the following holds.
\begin{enumerate}[({A}1)]
\item $d^i_{G,Q}(x)\leq \alpha n$.
\item $\overline{d}^i_{G,Q}(x) \leq \alpha n$.
\item There exists $z\in V(G)$ such that $|N^*_i(x,z)|+|N^*_i(z,x)| \leq \alpha n$.
\end{enumerate}
(Intuitively, the neighbourhood in $Q_i$ of an $i$-light vertex is `atypical', and this is unlikely to happen.)

Given $\psi>0$ and an index $i\in \{0,1,\dots, k-2\}$, we call $\{x,y_1,y_2,y_3\}\subseteq V(G)$ a {\em $(k,x,i,\psi)$-configuration} if it satisfies the following.
\begin{enumerate}[(C1)]
\item $G[\{x,y_1,y_2,y_3\}]$ is a linear forest.
\item $\overline{d}^{j}_{G,Q}(x) \geq 13\cdot 6^k \psi n$ for all $j\in \{0,1,\dots, k-2\} \setminus \{i\} $.
\item There exists $i'\neq i$ such that $d^{j}_{G,Q}(x) \geq 13 \cdot 6^k \psi n$ for all $j\in\{0,1,\dots,k-2\} \setminus \{i,i'\}$.
\item $\min\{ d^i_{G,Q}(y_j), \overline{d}^i_{G,Q}(y_j)\} \leq \psi^2 n$ for all $j\in [3]$.
\end{enumerate}
(Intuitively, (C1)--(C3) of the definition of $(k,x,i,\psi)$-configurations are useful for `building' induced copies of $C_{2k}$, so the existence of a $(k,x,i,\psi)$-configuration in an induced-$C_{2k}$-free graph $G$ severely constrains the choices for the remaining edge set of $G$. The bounds arising from this are still not sufficiently strong though; we also need (C4), which gives further constraints on the choices for the remaining edge set of $G$.)

We partition $F_Q(n,k,\eta,\mu)$ into the sets $T_Q, F^1_Q, F^2_Q, F^3_Q$ defined as follows.
\begin{enumerate}[(F1)]
\item[(F0)] $T_Q:= T_Q(n,k)\cap F_Q(n,k,\eta,\mu)$.
\item $F^1_Q \subseteq F_Q(n,k,\eta,\mu)\setminus T_Q$ is the set of all remaining graphs $G$ which satisfy one of the following.
\begin{enumerate}[{\rm (i)}]
\item $G$ contains a $(k,x,i,\psi)$-configuration for some $i\in \{0,1,\dots, k-2\}$, some $x\in V(G)$ and some $\psi \in \{\beta^{1/2}, \beta^2\}$.
\item $G$ contains a vertex $x$ which is both $i$-light and $j$-light for some distinct indices $i, j\in \{0,1,\dots, k-2\}$.
\end{enumerate}
\item $F^2_Q\subseteq F_Q(n,k,\eta,\mu) \setminus (T_Q \cup F^1_Q)$ is the set of all remaining graphs that for some $i\in \{0,1,\dots, k-2\}$ contain a vertex $x\in Q_i$ that satisfies $\overline{d}^{i}_{G,Q}(x), d^{i}_{G,Q}(x) \geq \beta n$.
\item $F^3_Q := F_Q(n,k,\eta,\mu) \setminus (T_Q \cup F^1_Q\cup F^2_Q)$ is the set of all remaining graphs.
\end{enumerate}

Sections~\ref{sec: estimate 1}, \ref{sec: estimate 2} and~\ref{sec: estimate 3} are devoted to proving upper bounds on $|F^1_Q|, |F^2_Q|$ and $|F^3_Q|$ respectively. As mentioned earlier, it turns out that $F^3_Q$ is the class of induced-$C_{2k}$-free graphs which are `extremely close' to being $k$-templates (see Proposition~\ref{beta}). In Section~\ref{sec: final calculation} we will use these bounds to complete the proof of Lemma~\ref{induction conclusion}.

\section{Estimation of $|F^1_Q|$}\label{sec: estimate 1}

To estimate $|F^1_Q|$ we will bound the number of graphs satisfying (F1)(i) and (F1)(ii) separately. The main difficulty is in estimating those satisfying (F1)(i), i.e. the ones containing a $(k,x,i,\psi)$-configuration. The idea here is that a $(k,x,i,\psi)$-configuration has many potential extensions into an induced copy of $C_{2k}$. More precisely, given a $(k,x,i,\psi)$-configuration $H$ we can find many disjoint `skeleton' graphs $L$ with the same number of components as $H$ such that $H\cup L$ is a linear forest on $2k$ vertices (i.e. $H\cup L$ has a potential extension into an induced $C_{2k}$). Thus each skeleton induces a restriction on further edges that can be added. Since the skeletons are disjoint we obtain many edge restrictions in total, and thus a good bound on the number of graphs containing a $(k,x,i,\psi)$-configuration. The next two propositions are used to formalise the notion of extendibility into an induced $C_{2k}$. (Roughly, in these propositions one can consider $L_1$ as a $(k,x,i,\psi)$-configuration and $L_2$ as an associated skeleton.)

\begin{proposition}\label{incomplete}
Let $c\geq 1$ and let $L_1, L_2$ be disjoint linear forests, each with exactly $c$ components, such that $|V(L_1)|+|V(L_2)|=2k$. Then there exists a set $E'$ of edges between $V(L_1)$ and $V(L_2)$ such that the graph $(V(L_1)\cup V(L_2),E'\cup E(L_1) \cup E(L_2))$ is isomorphic to $C_{2k}$.
\end{proposition}

The proof of Proposition~\ref{incomplete} is trivial, and is omitted\COMMENT{{\bf Proof of Proposition~\ref{incomplete}:} For every $j\in [2]$ we denote the components of $L_j$ by $P^1_j,\dots, P^c_j$. For every $i\in [c]$ and $j\in [2]$, if $|V(P^i_j)|>1$ we let the two endpoints of the path $P^i_j$ be denoted $s^i_j, t^i_j$. Otherwise $|V(P^i_j)|=1$ and we let $s^i_j = t^i_j$ be the unique vertex in $V(P^i_j)$. Then $E' := \{ t^1_1s^1_2, t^1_2s^2_1, t^2_1s^2_2, \dots, t^c_1s^c_2, t^c_2s^1_1 \}$ is a set of edges as required.}. Proposition~\ref{incomplete2} follows from an easy application of Proposition~\ref{incomplete}, and we give only a brief sketch of the proof.

\begin{proposition}\label{incomplete2}
Let $c\geq 1$ and let $L_1, L_2$ be linear forests that satisfy the following.
\begin{itemize}
\item $V(L_1)\cap V(L_2)=\{x\}$.
\item $|V(L_1)|, |V(L_2)| > 1$.
\item $d_{L_1}(x) + d_{L_2}(x)=2$.
\item $L_1$ and $L_2- \{x\}$ both have exactly $c$ components.
\item $|V(L_1)\cup V(L_2)| = 2k$.
\end{itemize}
Then there exists a set $E'$ of edges between $V(L_1)\backslash\{x\}$ and $V(L_2)\backslash\{x\}$ such that the graph $(V(L_1)\cup V(L_2),E'\cup E(L_1) \cup E(L_2))$ is isomorphic to $C_{2k}$.
\end{proposition}
\begin{proof}
If $d_{L_1}(x)=0$ we apply Proposition~\ref{incomplete} to $L_1-x, L_2$; if $d_{L_2}(x)=0$ we apply Proposition~\ref{incomplete} to $L_1, L_2-x$. If $d_{L_1}(x)=d_{L_2}(x)=1$ one can easily find $E'$ directly.
\end{proof}

\begin{lemma} \label{F^1}
$|F^1_Q| \leq C 2^{- \frac{\beta^2 n}{14^k}}  f_k(n_k) 2^{t_{k-1}(n)}$.
\end{lemma}
\begin{proof}
Let $F^1_{Q,(i)}$ denote the set of all graphs in $F^1_Q$ that satisfy (F1)(i). Similarly let $F^1_{Q,(ii)}$ denote the set of all graphs in $F^1_Q$ that satisfy (F1)(ii). Clearly,
\begin{equation}\label{eqn: F1 count}
|F^1_Q|\leq |F^1_{Q,(i)}| + |F^1_{Q,(ii)}|.
\end{equation}
We will first estimate the number of graphs in $F^1_{Q,(i)}$. Any graph $G\in F^1_{Q,(i)}$ can be constructed as follows. We first choose $\psi\in \{\beta^2, \beta^{1/2}\}$, and then perform the following steps.
\begin{itemize}
\item We choose an index $i\in \{0,1,\dots, k-2\}$, a set of three (labelled) vertices $Y=\{y_1,y_2,y_3\}$ in $[n]$, a vertex $x\in [n]\backslash Y$, and a set $E$ of edges between these four vertices such that $Y\cup \{x\}$ spans a linear forest. Let $b_1$ denote the number of such choices. The choices in the next steps will be made such that $Y\cup \{x\}$ is a $(k,x,i,\psi)$-configuration in $G$.
\item Next we choose the graph $G'$ on vertex set $[n]\backslash Y$ such that $G[[n]\backslash Y]=G'$. Let $b_2$ denote the number of possibilities for $G'$.
\item Next we choose the set $E'$ of edges in $G$ between $Y$ and $Q_i\backslash (Y\cup \{x\})$ such that $E'$ is compatible with our previous choices. Let $b_3$ denote the number of possibilities for $E'$.
\item Finally we choose the set $E''$ of edges in $G$ between $Y$ and $[n]\backslash (Q_i\cup Y \cup \{x\})$ such that $E''$ is compatible with our previous choices. Let $b_4$ denote the number of possibilities for $E''$.
\end{itemize}
Hence,
\begin{equation}\label{eqn: F1i count}
|F^1_{Q,(i)}|\leq 2\max\limits_{\psi\in \{\beta^2, \beta^{1/2}\}} \left\{b_1\cdot b_2\cdot b_3\cdot b_4\right\}.
\end{equation}

We then estimate the number of graphs in $F^1_{Q,(ii)}$. Any graph $G\in F^1_{Q,(ii)}$ can be constructed as follows.
\begin{itemize}
\item We first choose a single vertex $x$ from $[n]$ and distinct indices $i,j\in \{0,1,\dots, k-2\}$. Let $c_1$ denote the number of such choices. The choices in the next steps will be made such that $x$ is both $i$-light and $j$-light in $G$.
\item Next we choose the graph $G'$ on vertex set $[n]\backslash \{x\}$ such that $G[[n]\backslash \{x\}]=G'$. Let $c_2$ denote the number of possibilities for $G'$.
\item Next we choose the set $E$ of edges in $G$ between $\{x\}$ and $(Q_i\cup Q_j)\backslash \{x\}$ such that $E$ is compatible with our previous choices. Let $c_3$ denote the number of possibilities for $E$.
\item Finally we choose the set $E'$ of edges in $G$ between $\{x\}$ and $[n]\backslash (Q_i\cup Q_j\cup \{x\})$. Let $c_4$ denote the number of possibilities for $E'$.
\end{itemize}
Hence,
\begin{equation}\label{eqn: F1ii count}
|F^1_{Q,(ii)}|\leq c_1\cdot c_2\cdot c_3\cdot c_4.
\end{equation}
The following series of claims will give upper bounds for the quantities $b_1,\dots, b_4, c_1, \dots, c_4$. Claims~1 and~5 are trivial, while the proof of Claim~6 is almost identical to that of Claim~2; we give proofs of Claims~2,3,4,7 and~8.

\vspace{0.3cm}
\noindent {\bf Claim 1:} $b_1\leq 2^6kn^4$.

\vspace{0.3cm}
\noindent {\bf Claim 2:} $b_2\leq C 2^{\mu^{1/2}n} f_k(n_k) 2^{t_{k-1}(n-3)}$.

\noindent Indeed, note that for every graph $\tilde{G}\in F^1_{Q,(i)}$, Corollary~\ref{new partition} together with (\ref{size mathcal{P}}) implies that every optimal ordered $(k-1)$-partition of $\tilde{G}[[n]\backslash Y]$ is contained in some set $\mathcal{P}$ of size at most $2^{\mu n}$. Since $G[[n]\backslash Y]$ is clearly induced-$C_{2k}$-free, this together with (\ref{induct}) implies that
\begin{align*}
b_2 &\leq \sum_{Q'\in \mathcal{P}} |F_{Q'}(n-3,k)| \leq 6C  2^{\mu n} 2^{6(\log n)^2} f_k(\lceil (n-3)/(k-1) \rceil ) 2^{t_{k-1}(n-3)}\\
&\leq C 2^{\mu^{1/2}n} f_k(n_k) 2^{t_{k-1}(n-3)},
\end{align*}
as required.

\vspace{0.3cm}
\noindent {\bf Claim 3:} $b_3\leq 2^{4\psi^{3/2} n}$.

\noindent Indeed, for every graph $\tilde{G}\in F^1_{Q,(i)}$ for which $\{x,y_1,y_2,y_3\}$ is a $(k,x,i,\psi)$-configuration we have that $\min\{d^i_{\tilde{G},Q}(y_j), \overline{d}^i_{\tilde{G},Q}(y_j) \} \leq \psi^2 n$ for all $j\in [3]$. So $b_3\leq \prod_{j=1}^{3} h(j)$ where $h(j)$ denotes the number of possibilities for a set of edges between $\{y_j\}$ and $Q_i\backslash (Y\cup \{x\})$ such that either $d^i_{G,Q}(y_j)\leq \psi^2 n$ or $\overline{d}^i_{G,Q}(y_j)\leq \psi^2 n$. Note that by (\ref{entropy bound}), $h(j)\leq 2 {\binom{n}{\leq \psi^2 n}} \leq 2^{\xi(\psi^2) n+1}$. Hence,
$$b_3\leq \prod\limits_{j=1}^{3} h(j)\leq ( 2^{\xi(\psi^2)n+1})^3 \stackrel{(\ref{entropy bound 1.5})}{\leq} 2^{4\psi^{3/2} n},$$
as required.

\vspace{0.3cm}
\noindent {\bf Claim 4:} $b_4\leq 2^{3(k-2)n/(k-1)}2^{\mu^{1/2} n} 2^{-\psi n/11^k}$.

\noindent Indeed, first define $L$ to be the graph on vertex set $Y\cup \{x\}$ that satisfies $E(L)=E$. We say an induced subgraph $H$ of $G'-x$ is an $L$-\emph{compatible skeleton} if it satisfies the following.

\begin{itemize}
\item $|V(H)|=2k-4$.
\item $G'[V(H)\cup \{x\}]$ is a linear forest.
\item In $G'$, $x$ has $2-d_L(x)$ neighbours in $V(H)$.
\item $L$ and $H$ have the same number of components.
\end{itemize}

Given an $L$-compatible skeleton $H$, note that Proposition~\ref{incomplete2}, applied with $L, G'[V(H)\cup \{x\}]$ playing the roles of $L_1, L_2$ respectively, implies that there exists a set $E_{L,H}$ of possible edges between $Y$ and $V(H)$ such that $(Y\cup \{x\}\cup V(H), E\cup E(H)\cup E_{L,H})$ is isomorphic to $C_{2k}$.

We will show that there exist a large number of disjoint $L$-compatible skeletons in $G'-x$. Since there is a limited number of ways to choose edges between $Y$ and each of these $L$-compatible skeletons so as not to create an induced copy of $C_{2k}$, this will imply the claim.

For every index $j\ne i$, let $N^1_j(x), N^2_j(x)\subseteq N_{Q_j}(x)$ be disjoint with $|N^1_j(x)|,|N^2_j(x)|\geq \lfloor \frac{1}{2}|N_{Q_j}(x)| \rfloor$. Similarly, let $\overline{N}^1_j(x), \overline{N}^2_j(x)\subseteq \overline{N}_{Q_j}(x)$ be disjoint with $|\overline{N}^1_j(x)|,|\overline{N}^2_j(x)|\geq \lfloor \frac{1}{2}|\overline{N}_{Q_j}(x)| \rfloor$.

Note that we may assume that there exists an index $i'\in \{0,1,\dots, k-2\}\backslash \{i\}$ such that in $G'$, $|\overline{N}_{Q_j}(x)|\geq 12 \cdot 6^k \psi n$ for all $j\in \{0,1,\dots, k-2\}\backslash \{i\}$ and $|N_{Q_j}(x)| \geq 12 \cdot 6^k \psi n$ for all $j\in \{0,1,\dots, k-2\}\backslash \{i,i'\}$, since otherwise $\{x,y_1,y_2,y_3\}$ cannot be a $(k,x,i,\psi)$-configuration. Define $\ell_1,\dots,\ell_{k-2}$ such that $\{\ell_1,\dots, \ell_{k-2}\}=\{0,1,\dots, k-2\}\backslash \{i\}$ and $\ell_{k-2} = i'$. Thus the following hold.
\begin{enumerate}[(a)]
\item $|N^1_{\ell_{j}}(x)|,|N^2_{\ell_{j}}(x)|, |\overline{N}^1_{\ell_{j}}(x)|, |\overline{N}^2_{\ell_{j}}(x)|\geq 6 \cdot 6^k \psi n$ for all $j\in \{1,\dots, k-3\}$.
\item $|\overline{N}^1_{\ell_{k-2}}(x)|, |\overline{N}^2_{\ell_{k-2}}(x)|\geq 6 \cdot 6^k \psi n$.
\end{enumerate}

We now show that $G'-x$ contains at least $5\cdot 6^k\psi n$ disjoint $L$-compatible skeletons. Define $t$ to be the number of components of $L$, and define $s:=d_L(x)$. Then $1\leq t\leq 4$ and $0\leq s\leq 2$. Note that $t+s\geq 2$, since a $4$-vertex linear forest with one component contains no isolated vertices. We consider two cases. In each case we will describe the length and type of $t$ path components, $P^1,\dots, P^t$, each with an even number of vertices. Proposition~\ref{building} (applied repeatedly) together with (a),(b) will then imply that $G'-x$ contains at least $5\cdot 6^k\psi n$ disjoint $L$-compatible skeletons, each consisting exactly of $t$ components isomorphic to $P^1,\dots, P^t$. (We can apply Proposition~\ref{building} here since in each case $P^1 \cup \dots \cup P^t$ will contain a perfect matching.)

\begin{enumerate}[{\bf {Case} 1:}]
\item \emph{$s=2$.}

\begin{itemize}

\item For $1\leq r\leq t-1$, $P^r$ is a $K_2$ of type $\overline{N}^1_{{\ell_r}}(x), \overline{N}^2_{{\ell_r}}(x)$.

\item $P^{t}$ is a $P_{2k-2t-2}$ of type $ \overline{N}^1_{{\ell_t}}(x), \overline{N}^2_{{\ell_t}}(x), \overline{N}^1_{{\ell_{t+1}}}(x), \overline{N}^2_{{\ell_{t+1}}}(x),\dots,  \overline{N}^1_{{\ell_{k-2}}}(x), \overline{N}^2_{{\ell_{k-2}}}(x)$.
\end{itemize}

\item \emph{Either $s=1$ or $s=0, t>1$.}

\begin{itemize}
\item For $1\leq r\leq 1-s$, $P^r$ is a $K_2$ of type $N^1_{{\ell_r}}(x), \overline{N}^1_{{\ell_r}}(x)$.

\item $P^{2-s}$ is a $P_{2k-2t-2}$ of type $ N^1_{{\ell_{2-s}}}(x), \overline{N}^2_{{\ell_{2-s}}}(x), \overline{N}^1_{{\ell_{3-s}}}(x), \overline{N}^2_{{\ell_{3-s}}}(x),\dots,  \overline{N}^1_{{\ell_{k-t-s}}}(x),\\ \overline{N}^2_{{\ell_{k-t-s}}}(x)$.

\item For $k-t-s+1\leq r\leq k-2$, $P^r$ is a $K_2$ of type $\overline{N}^1_{{\ell_r}}(x), \overline{N}^2_{{\ell_r}}(x)$.
\end{itemize}
\end{enumerate}

Since $t+s\geq 2$, this covers all cases. Now fix a set $SK$ of $5\cdot 6^k\psi n$ disjoint $L$-compatible skeletons in $G'-x$, and let $H\in SK$. Let $h_H$ denote the number of possibilities for a set $E^*$ of edges between $Y$ and $V(H)$. Note that such a set $E^*$ cannot equal $E_{L,H}$, since $G$ needs to be induced-$C_{2k}$-free. Thus $h_H\leq 2^{|Y||V(H)|}-1=2^{6(k-2)}-1$. Note that by $({\rm F}2)_{\mu}$ the number of vertices outside $Q_i$ that are not contained in some graph $H\in SK$ is at most $(k-2)n/(k-1) + \mu n - 10(k-2)6^k \psi n$. Hence,
\begin{align*}
b_4&\leq 2^{3(k-2)n/(k-1) - 30(k-2)6^k\psi n + 3\mu n}\prod\limits_{H\in SK} h_H \\
&\leq 2^{3(k-2)n/(k-1) - 30(k-2)6^k\psi n + 3\mu n} \left( 2^{6(k-2)}\left( 1-2^{-6(k-2)} \right) \right)^{5\cdot 6^k\psi n}\\
&\leq 2^{3(k-2)n/(k-1)}2^{3\mu n} e^{-5\cdot 6^k\psi n/(2^{6(k-2)})}\leq 2^{3(k-2)n/(k-1)}2^{\mu^{1/2} n} 2^{-\psi n/11^k},
\end{align*}
as required.

\vspace{0.3cm}
\noindent {\bf Claim 5:} $c_1\leq k^2 n$.

\vspace{0.3cm}
\noindent {\bf Claim 6:} $c_2\leq C2^{\mu^{1/2}n} f(n_k) 2^{t_{k-1}(n-1)}$.

\COMMENT{Indeed, note that for every graph $\tilde{G}\in F^1_{Q,(ii)}$, Corollary~\ref{new partition} together with (\ref{size mathcal{P}}) implies that every optimal ordered $(k-1)$-partition of $\tilde{G}[[n]\backslash \{x\}]$ is contained in some set $\mathcal{P}$ of size at most $2^{\mu n}$. Since $G[[n]\backslash \{x\}]$ is clearly induced-$C_{2k}$-free, this together with (\ref{induct}) implies that
\begin{align*}
c_2 &\leq \sum_{Q'\in \mathcal{P}} |F_{Q'}(n-1,k)| \leq 6C2^{\mu n} 2^{6(\log n)^2} f(\lceil (n-1)/(k-1) \rceil ) 2^{t_{k-1}(n-1)}\\
&\leq C2^{\mu^{1/2}n} f(n_k) 2^{t_{k-1}(n-1)},\end{align*}
as required.}

\vspace{0.3cm}
\noindent {\bf Claim 7:} $c_3\leq 2^{7\xi(\alpha)n}$.

\noindent Indeed, for every graph $\tilde{G}\in F^1_{Q,(ii)}$ for which $x$ is both $i$-light and $j$-light, we have that, for every $\ell\in \{i,j\}$, either $\min \{|N_{Q_{\ell}}(x)|, |\overline{N}_{Q_{\ell}}(x)|\}\leq \alpha n$ or else there exists a vertex $z\neq x$ such that $|N^*_{\ell}(x,z)| + |N^*_{\ell}(z,x)| \leq \alpha n$.

For $\ell\in \{i,j\}$, let $h(\ell,1)$ denote the number possibilities for a set of edges in $G$ between $\{x\}$ and $Q_{\ell}\backslash \{x\}$ such that $\min \{|N_{Q_{\ell}}(x)|, |\overline{N}_{Q_{\ell}}(x)|\}\leq \alpha n$. Then $h(\ell,1)\leq 2\binom{n}{\leq \alpha n}\leq 2^{\xi(\alpha)n+1}$. For $\ell\in \{i,j\}$, let $h(\ell,2)$ denote the number possibilities for a set of edges between $\{x\}$ and $Q_{\ell}\backslash \{x\}$ such that there exists a vertex $z\neq x$ such that $|N^*_{\ell}(x,z)| + |N^*_{\ell}(z,x)|  \leq \alpha n$. Then $h(\ell,2) \leq n{\binom{|N_{Q_{\ell}}(z)|}{\leq \alpha n}}{\binom{|\overline{N}_{Q_{\ell}}(z)|}{\leq \alpha n}}\leq 2^{3\xi(\alpha)n}$.

Hence
$$c_3\leq (h(i,1)+h(i,2))(h(j,1)+h(j,2))\leq (2^{\xi(\alpha)n+1}+ 2^{3\xi(\alpha)n})^2 \leq 2^{7\xi(\alpha)n},$$
as required.

\vspace{0.3cm}
\noindent {\bf Claim 8:} $c_4\leq 2^{ (k-3)n/(k-1) } 2^{2\mu n}$.

\noindent Indeed, since the number of possible edges between $\{x\}$ and $[n]\backslash (Q_i\cup Q_j\cup \{x\})$ is at most $(k-3)n/(k-1) + 2\mu n$, we have that $c_4 \leq 2^{ (k-3)n/(k-1) + 2\mu n}$, as required.

\vspace{0.3cm}
\noindent Now (\ref{eqn: F1i count}) together with Claims~1--4 and Proposition~\ref{omitted proposition}(ii) implies that
\begin{align}\label{eqn: F1i count two}
&|F^1_{Q,(i)}|\\
\leq \hspace{0.2cm} &2\max\limits_{\psi\in \{\beta^2, \beta^{1/2}\}} \left\{2^6kn^4\cdot C 2^{\mu^{1/2}n} f_k\left( n_k \right) 2^{t_{k-1}(n-3)}\cdot 2^{4\psi^{3/2} n}\cdot 2^{\frac{3(k-2)n}{k-1}}2^{\mu^{1/2} n} 2^{-\frac{\psi n}{11^k}} \right\} \nonumber\\
\leq \hspace{0.2cm} &\max\limits_{\psi\in \{\beta^2, \beta^{1/2}\}} \left\{ Cf_k(n_k)2^{t_{k-1}(n-3)+\frac{3(k-2)n}{k-1}}2^{-\frac{\psi n}{12^k}} \right\} \leq Cf_k(n_k)2^{t_{k-1}(n)}2^{-\frac{\beta^2 n}{13^k}}.\nonumber
\end{align}
Similarly, (\ref{eqn: F1ii count}) together with Claims~5--8 and Proposition~\ref{omitted proposition}(ii) implies that
\begin{align}\label{eqn: F1ii count two}
|F^1_{Q,(ii)}|&\leq k^2 n\cdot C2^{\mu^{1/2}n} f_k\left( n_k \right) 2^{t_{k-1}(n-1)}\cdot 2^{7\xi(\alpha)n}\cdot 2^{ \frac{(k-3)n}{k-1} } 2^{2\mu n}\\
&\leq C2^{\mu^{1/3}n}f_k(n_k)2^{t_{k-1}(n-1) + \frac{(k-2)n}{k-1}}2^{-\frac{n}{k-1}}2^{7\xi(\alpha)n}\leq Cf_k(n_k)2^{t_{k-1}(n)}2^{-\frac{n}{k}}.\nonumber
\end{align}
Now (\ref{eqn: F1 count}) together with (\ref{eqn: F1i count two}) and (\ref{eqn: F1ii count two}) implies that
$$|F^1_Q|\leq Cf_k(n_k)2^{t_{k-1}(n)}\left(2^{-\frac{\beta^2 n}{13^k}}+2^{-\frac{n}{k}}\right)\leq Cf_k(n_k)2^{t_{k-1}(n)}2^{-\frac{\beta^2 n}{14^k}},$$
as required.
\end{proof}

\section{Estimation of $|F^2_Q|$}\label{sec: estimate 2}

Given $G\in F^2_Q\cup F^3_Q$ and $i\in \{0,1,\dots, k-2\}$, let $A^i_G := \{x \in Q_i: \overline{d}^i_{G,Q}(x),d^i_{G,Q}(x) \geq \beta n\}$. The key result of this section (Lemma~\ref{size of A^i_G}) states that $A^i_G$ has bounded size. To prepare for this, we will classify the pairs of vertices in $A^i_G$ according to their (non-)neighbourhood intersection pattern. The fact that $G\notin F^1_Q$ allows us to observe some restrictions on these patterns (see Propositions~\ref{prop: not j j' identical} and~\ref{asym-reg}). In the proof of Lemma~\ref{size of A^i_G} we use a Ramsey argument to restrict our view to one abundant type of pattern. This quickly leads to a contradiction if $|A^i_G|$ is large. Using the fact that $G\notin F^1_Q$ we show that the remainder of each class (i.e. $G[Q_i\backslash A^i_G]$) induces a very simple structure (Proposition~\ref{Q_i setminus A^i_G}). We translate this structural information into a sufficiently strong bound on the number of graphs in $F^2_G$, in Lemma~\ref{F^2_Q}.

Let $\mathcal{L}$ denote the collection of all $4$-vertex linear forests. The following proposition is an analogue of Proposition~\ref{prop: k-good tetradehron free}(i) that can be applied to graphs rather than $2$-coloured multigraphs. It follows immediately from Proposition~\ref{prop: k-good tetradehron free}(i).

\begin{proposition}\label{prop: L free}
Let $G$ be a graph such that for every $H\in \mathcal{L}$, $G$ is induced $H$-free. Then $\overline{G}$ is a disjoint union of stars and triangles.
\end{proposition}

\begin{proposition}\label{Q_i setminus A^i_G}
Let $G\in F^2_Q\cup F^3_Q$ and $i\in \{0,1,\dots, k-2\}$. Then $\overline{G}[Q_i\setminus A^i_G]$ is a disjoint union of stars and triangles.
\end{proposition}
\begin{proof}
Suppose for a contradiction that $\overline{G}[Q_i\setminus A^i_G]$ is not a disjoint union of stars and triangles. Then Proposition~\ref{prop: L free} implies that $\overline{G}[Q_i\setminus A^i_G]$ contains an induced copy of a graph in $\mathcal{L}$, with vertex set $\{x,y_1,y_2,y_3\}$ say. We will show that $\{x,y_1,y_2,y_3\}$ is a $(k,x,i,\beta^{1/2})$-configuration, which contradicts the fact that $G\notin F^1_Q$. Note that $G[\{x,y_1,y_2,y_3\}]$ is a linear forest, and so $\{x,y_1,y_2,y_3\}$ satisfies (C1). By the definition of $A^i_G$ we have that $\min \{d^i_{G,Q}(y_j), \overline{d}^i_{G,Q}(y_j)\}\leq \beta n$ for all $j\in [3]$, and so $\{x,y_1,y_2,y_3\}$ satisfies (C4). Since $G\notin F^1_Q$, $x$ is $j$-light for at most one index $j\in\{0,1,\dots, k-2\}$. Since $x \in Q_i\setminus A^i_G$, $x$ is $i$-light. Thus for every $j\in \{0,1,\dots, k-2\}$ with $i\ne j$ we have that $x$ is not $j$-light, and hence $d^j_{G,Q}(x), \overline{d}^j_{G,Q}(x) > \alpha n > 13\cdot 6^k \cdot \beta^{1/2} n$, and so $\{x,y_1,y_2,y_3\}$ satisfies (C2) and (C3). Therefore $\{x,y_1,y_2,y_3\}$ is a $(k,x,i,\beta^{1/2})$-configuration, as required.
\end{proof}

The following definitions will be useful in order to show that $|A^i_G|$ is small. Suppose $S$ is a star or triangle. If $S$ is a star on at least three vertices, we call the unique vertex in $S$ of degree greater than one the \emph{centre} of $S$. Otherwise we call the vertex of $S$ with the smallest label the centre of $S$.

Let $G\in F^2_Q\cup F^3_Q$ and $i,j\in \{0,1,\dots, k-2\}$ and let $x,y\in A^i_G$.
\begin{itemize}
\item We say $x,y$ are \emph{$j$-irregular} if $|\overline{N}_{j}(\{x,y\})| \leq \gamma n$.

\item We say $x,y$ are \emph{$j$-asymmetric} if $|N^*_j(x,y)| +|N^*_j(y,x)| > 3\gamma n$ and either $|N^*_j(x,y)| \leq \gamma n$ or $|N^*_j(y,x)| \leq \gamma n$.

\item We say $x,y$ are \emph{$j$-identical} if $|N^*_j(x,y)| +|N^*_j(y,x)|  \leq 3\gamma n$.
\end{itemize}
Roughly speaking, if one of the above holds then the neighbourhoods of $x,y$ do not behave in a `random' like way (thus constraining the number of possibilities for choosing the neighbourhoods). The following statement follows immediately from the above definitions and the fact that $\gamma \ll \alpha$.
\begin{equation}\label{identical-light}
\text{If }x, y \text{ are } j \text{-identical then }x,y \text{ are both }j\text{-light.}
\end{equation}

\begin{proposition}\label{prop: not j j' identical}
Let $G\in F^2_Q\cup F^3_Q$ and $i\in \{0,1,\dots, k-2\}$ and let $x,y \in A^i_G$. Then $x,y$ are $j$-identical for at most one index $j\in \{0,1,\dots, k-2\}$.
\end{proposition}
\begin{proof}
Suppose $x,y$ are $j$-identical for some $j\in \{0,1,\dots, k-2\}$ and suppose $j'\in \{0,1,\dots, k-2\}$ with $j'\ne j$. It suffices to show that $x,y$ are not $j'$-identical. Note that $x$ is $j$-light by (\ref{identical-light}). Since $G\notin F^1_Q$, $x$ is $j''$-light for at most one index $j''\in \{0,1,\dots, k-2\}$. Thus $x$ is not $j'$-light, and hence by (\ref{identical-light}) $x,y$ are not $j'$-identical, as required.
\end{proof}

\begin{proposition} \label{asym-reg}
Let $G\in F^2_Q\cup F^3_Q$ and $i\in \{0,1,\dots, k-2\}$ and let $x,y \in A^i_G$. Then there exists an index $j\in \{0,1,\dots, k-2\}$ such that $x,y$ are $j$-irregular or $j$-asymmetric (or both).
\end{proposition}
\begin{proof}
Suppose for a contradiction that for every index $\ell\in \{0,1,\dots, k-2\}$, $x,y$ are neither $\ell$-irregular nor $\ell$-asymmetric. Since, by Proposition~\ref{prop: not j j' identical}, $x,y$ are $j$-identical for at most one index $j$, and $k\geq 6$, we may assume without loss of generality that $x,y$ are not $\ell$-identical for $\ell\in \{1,2,3\}$. We consider the following two cases.

\vspace{0.3cm}
\noindent {\bf Case 1:} \emph{$x,y$ are adjacent.}

\noindent In this case we define sets $A_{\ell}, B_{\ell}$ for $\ell\in \{0,1,\dots, k-2\}$ as follows. We will use these sets to extend $x,y$ into an induced copy of $C_{2k}$.
\begin{itemize}
\item Let $A_1:= N^*_1(x,y)$ and $B_1 := \overline{N}_1(\{x,y\})$.
\item Let $A_2:= N^*_2(y,x)$ and $B_2 := \overline{N}_2(\{x,y\})$.
\item For every $\ell\in \{0,1,\dots, k-2\}\backslash \{1,2\}$, let $A_{\ell}, B_{\ell}\subseteq \overline{N}_{\ell}(\{x,y\})$ be disjoint and satisfy $|A_{\ell}|,|B_{\ell}|\geq \lfloor|\overline{N}_{\ell}(\{x,y\})|/2\rfloor$.
\end{itemize}
Since for every $\ell\in \{0,1,\dots, k-2\}$ $x,y$ are neither ${\ell}$-irregular nor ${\ell}$-asymmetric, and for every $\ell\in \{1,2\}$ $x,y$ are not $\ell$-identical, we have that $|A_{\ell}|,|B_{\ell}|\geq \gamma n/3$ for every $\ell\in \{0,1,\dots, k-2\}$. This together with Proposition~\ref{building} and the fact that $\mu \ll \gamma$ implies that there exists in $G$ an induced copy of $P_{2k-2}$ of type $A_1,B_1,A_0,B_0,A_3,B_3,\dots, A_{k-2},B_{k-2}, B_2,A_2$. By the definition of the sets $A_{\ell}, B_{\ell}$, the vertices of this $P_{2k-2}$ together with $x,y$ induce on $G$ a copy of $C_{2k}$. This contradicts the fact that $G\in F_Q(n,k)$.
\vspace{0.3cm}

\noindent {\bf Case 2:} \emph{$x,y$ are not adjacent.}

\noindent In this case we define sets $A_{\ell}, B_{\ell}$ for $\ell\in \{0,1,\dots, k-2\}$ as follows. Similarly to the previous case, we will find an induced $C_{2k}$ which contains $x,y$ together with exactly one vertex from each of these sets.
\begin{itemize}
\item Let $A_1:= N^*_1(x,y)$ and $B_1 := N^*_1(y,x)$.
\item Let $A_2:= N^*_2(x,y)$ and $B_2 := \overline{N}_2(\{x,y\})$.
\item Let $A_3:= N^*_3(y,x)$ and $B_3 := \overline{N}_3(\{x,y\})$.
\item For every $\ell\in \{0,1,\dots, k-2\}\backslash \{1,2,3\}$, let $A_{\ell}, B_{\ell}\subseteq \overline{N}_{\ell}(\{x,y\})$ be disjoint and satisfy $|A_{\ell}|,|B_{\ell}|\geq \lfloor|\overline{N}_{\ell}(\{x,y\})|/2\rfloor$.
\end{itemize}
Since for every $\ell\in \{0,1,\dots, k-2\}$ $x,y$ are neither ${\ell}$-irregular nor ${\ell}$-asymmetric, and for every $\ell\in \{1,2,3\}$ $x,y$ are not $\ell$-identical, we have that $|A_{\ell}|,|B_{\ell}|\geq \gamma n/3$ for every $\ell\in \{0,1,\dots, k-2\}$. As before, this together with Proposition~\ref{building} implies that there exists in $G$ an induced copy of the graph $H$ that consists of the following two components:
\begin{itemize}
\item One $P_{2k-4}$ of type $A_2,B_2,A_0,B_0,A_4,B_4,\dots, A_{k-2},B_{k-2}, B_3,A_3$.
\item One $K_2$ of type $A_1,B_1$.
\end{itemize}
By the definition of the sets $A_{\ell}, B_{\ell}$, the vertices of $H$ together with $x,y$ induce on $G$ a copy of $C_{2k}$. This contradicts the fact that $G\in F_Q(n,k)$.
\vspace{0.3cm}

\noindent This covers all cases, and hence completes the proof.
\end{proof}

Recall from Section~\ref{sec: main proof} that $M:= R_{2k-2}(\lceil \frac{1}{\gamma}\rceil) +1$.

\begin{lemma}\label{size of A^i_G}
Let $G\in F^2_Q\cup F^3_Q$ and $i\in \{0,1,\dots, k-2\}$. Then $|A^i_G| < M$.
\end{lemma}
\begin{proof}
Suppose for a contradiction that $|A^i_G| \geq M$. Consider an auxiliary complete graph $H_i$ with $V(H_i)= A^i_G$. We define a $(2k-2)$-edge-colouring $\mathcal{C}$ of $H_i$ with colours $\{a_0,b_0,a_1,b_1,\dots, a_{k-2},b_{k-2}\}$ as follows.
\begin{itemize}
\item For every $j\in \{0,1,\dots, k-2\}$, an edge $xy\in E(H)$ is coloured $a_j$ if $x,y$ are $j$-irregular and for every $j'\in \{0,1,\dots, k-2\}$ with $j'<j$, $x,y$ are not $j'$-irregular.
\item An edge $xy\in E(H)$ that was not coloured in the previous step is coloured $b_j$ if $x,y$ are $j$-asymmetric, and for every $j'\in \{0,1,\dots, k-2\}$ with $j'<j$, $x,y$ are not $j'$-asymmetric.
\end{itemize}
Note that by Proposition~\ref{asym-reg}, every edge is coloured by a unique colour in $\mathcal{C}$.

Now since $M > R_{2k-2}(\lceil 1/\gamma \rceil )$, $H_i$ contains a monochromatic clique of size at least $1/\gamma$. Let $X=\{x_1,x_2,\dots, x_{\lceil 1/\gamma \rceil}\}$ be the vertex set of such a monochromatic clique. We consider the following two cases.

\vspace{0.3cm}

\noindent {\bf Case 1:} \emph{$X$ has colour $a_j$ for some $j\in \{0,1,\dots, k-2\}$.}

\noindent In this case every pair of vertices in $X$ is $j$-irregular, by definition of $\mathcal{C}$. Let $X':=\{x_1,x_2,\dots, x_{\lceil \beta/2\gamma \rceil} \}$ and suppose $z,z'\in X'$. By the definition of $j$-irregularity, $|\overline{N}_{j}(z)\cap\overline{N}_{j}(z')|\leq \gamma n$. Note also that $|\overline{N}_{j}(z)| \geq \beta n$ by Proposition~\ref{optimality} and the fact that $z\in A^i_G$. So by the inclusion-exclusion principle,
\begin{align*}
2n/(k-1)&\geq n/(k-1) + \mu n \geq |Q_j| \geq \sum_{z\in X'} |\overline{N}_{j}(z)| - \sum_{\stackrel{z,z'\in X'}{z \ne z'}} |\overline{N}_{j}(z)\cap\overline{N}_{j}(z')|\\
&\geq \beta \lceil \beta/2\gamma \rceil n - \lceil \beta^2/(4\gamma^2)\rceil \gamma n \geq \beta^2n/5\gamma > 2n/(k-1),
\end{align*}
where the last inequality follows from the fact that $\gamma \ll \beta$. This is a contradiction.

\vspace{0.3cm}

\noindent {\bf Case 2:} \emph{$X$ has colour $b_j$ for some $j\in \{0,1,\dots, k-2\}$.}

\noindent In this case every pair of vertices in $X$ is $j$-asymmetric, by definition of $\mathcal{C}$. Suppose $\ell,\ell'\in [\lceil 1/\gamma \rceil]$ are distinct. By the definition of $j$-asymmetry, exactly one of the following holds.
\begin{enumerate}[(a)]
\item $|N^*_j(x_\ell,x_{\ell'})| \leq \gamma n$ and $|N^*_j(x_{\ell'},x_{\ell})| > 2\gamma n$.
\item $|N^*_j(x_{\ell'},x_{\ell})| \leq \gamma n$ and $|N^*_j(x_{\ell},x_{\ell'})| > 2\gamma n$.
\end{enumerate}
Consider the auxiliary tournament $T$ with $V(T)=X$ and $E(T)= \{ \overrightarrow{x_\ell x_{\ell'}} : \ell,\ell' \text{ satisfy (a)}\}$. By Redei's theorem every tournament contains a directed Hamilton path. So, by relabelling the indices if necessary, we may assume that $\overrightarrow{x_{\ell}x_{\ell+1}}\in E(T)$ for every $\ell\in [\lceil 1/\gamma \rceil-1]$. Thus for every $\ell\in [\lceil 1/\gamma \rceil-1]$,
\begin{align*}
|\overline{N}_{j}(x_{\ell+1})| = |\left(\overline{N}_{j}(x_{\ell}) \setminus N^*_j(x_{\ell+1},x_\ell) \right) \cup N^*_{j}(x_\ell,x_{\ell+1})|\leq |\overline{N}_{j}(x_{\ell})| -2\gamma n + \gamma n \leq |\overline{N}_{j}(x_{\ell})| -\gamma n.
\end{align*}
Hence,
$$|\overline{N}_{j}(x_{\lceil 1/\gamma \rceil})| \leq  |\overline{N}_{j}(x_{1})| - \left( \frac{1}{\gamma}-1\right) \cdot \gamma n \leq  |Q_j| - (1-\gamma)n < 0,$$
which is a contradiction.

\vspace{0.3cm}

\noindent This covers all cases, and hence completes the proof.
\end{proof}

Suppose $G\in F^2_Q$ and $i\in \{0,1,\dots, k-2\}$. By Proposition~\ref{Q_i setminus A^i_G}, $\overline{G}[Q_i \setminus A^i_G]$ is a disjoint union of stars and triangles. Let $\mathcal{S}$ be the set of components of $\overline{G}[Q_i \setminus A^i_G]$ with the largest number of vertices. Let $S^{\diamond}$ be the component in $\mathcal{S}$ whose centre $c$ has the smallest label. Define $Y_i=Y_i(G,Q)$ to be the set of all isolated vertices in $\overline{G}[Q_i\setminus A^i_G]$ together with all vertices in $V(S^{\diamond})\backslash \{c\}$.

\begin{lemma}\label{size of Y_i}
Let $G\in F^2_Q$ and $i\in \{0,1,\dots, k-2\}$. Then $|Y_i| \geq 10 n/\log n$.
\end{lemma}
\begin{proof}
Define $s:=\lceil 10 n/\log n \rceil$. Suppose for a contradiction that $|Y_i|< s$. Since $G \in F^2_Q$, there exists an index $i'\in \{0,1,\dots, k-2\}$ such that $|A^{i'}_G|>0$. Let $x\in A^{i'}_G$. The definition of $A^{i'}_G$ together with Proposition~\ref{optimality} implies that $|\overline{N}_{Q_{j}}(x)|\geq \beta n$ for every $j\in \{0,1,\dots, k-2\}$. This together with Lemma~\ref{size of A^i_G} implies that $|\overline{N}_{Q_i}(x)\setminus A^i_G| \geq \beta n - M > 2s$. Also, since $|Y_i|<s$, at most $s$ components in $\overline{G}[Q_i\setminus A^i_G]$ are isolated vertices and every component in $\overline{G}[Q_i\setminus A^i_G]$ has order at most $s$. Thus there are at least two non-trivial components $S,S'$ of $\overline{G}[Q_i\setminus A^i_G]$ that each contain a non-neighbour of $x$.

Since $S$ is a non-trivial component of $\overline{G}[Q_i\setminus A^i_G]$ there exist vertices $y,y'\in S$ such that $xy,yy'\notin E(G[Q_i])$. Let $y''\in S'$ be such that $xy''\notin E(G[Q_i])$. Since $y''$ belongs to a different component of $\overline{G}[Q_i\setminus A^i_G]$ to $y$ and $y'$, it follows that $yy'',y'y''\in E(G[Q_i])$. Thus,
\begin{equation}\label{linear forest}
E(G[\{x,y,y',y''\}])\in\{ \{yy'',y'y''\}, \{xy',yy'',y'y''\} \}.
\end{equation}

\vspace{0.3cm}

\noindent {\bf Claim:} \emph{$\{x,y,y',y''\}$ is a $(k,x,i,\beta^2)$-configuration.}

\noindent Indeed, by (\ref{linear forest}), $G[\{x,y,y',y''\}]$ is a linear forest and so $\{x,y,y',y''\}$ satisfies (C1). As observed above, $\overline{d}^j_{G,Q}(x)\geq \beta n > 13\cdot 6^k \beta^2 n$ for every $j\in \{0,1,\dots, k-2\}$, and so $\{x,y,y',y''\}$ satisfies (C2). Since $G\notin F^1_Q$, there do not exist distinct $j,j'\in \{0,1,\dots, k-2\}$ such that $x$ is both $j$-light and $j'$-light. So there exists $j\in \{0,1,\dots, k-2\}$ such that for every $j'\in \{0,1,\dots, k-2\}$ with $j'\ne j$, $d^{j'}_{G,Q}(x)> \alpha n > 13\cdot 6^k \beta^2 n$, and so $\{x,y,y',y''\}$ satisfies (C3). Since $S,S'$ each contain at most $s$ vertices, $y,y',y''$ each have at most $s$ non-neighbours in $G[Q_i\setminus A^i_G]$. This together with Lemma~\ref{size of A^i_G} implies that $y,y',y''$ each have at most $s+M\leq \beta^4 n$ non-neighbours in $G[Q_i]$, and so $\{x,y,y',y''\}$ satisfies (C4). Hence $\{x,y,y',y''\}$ is a $(k,x,i,\beta^2)$-configuration, as required.

\vspace{0.3cm}

\noindent The above claim contradicts the fact that $G\notin F^1_Q$, and hence completes the proof.
\end{proof}

Lemma~\ref{size of Y_i} guarantees a large set of vertices in each class $Q_i$ (namely $Y_i$) with an extremely restricted (non-)neighbourhood. This is the key idea in our estimation of $|F^2_Q|$.

\begin{lemma}\label{F^2_Q}
$ |F^2_Q| \leq C 2^{-n}f_k(n_k) 2^{t_{k-1}(n)}$.
\end{lemma}
\begin{proof}
Define $s:=\lceil 10 n/\log n \rceil$. Since by Lemma~\ref{size of Y_i} $|Y_i(G,Q)|\geq s$ for every graph $G\in F^2_Q$, any graph $G\in F^2_Q$ can be constructed as follows.
\begin{itemize}
\item First we choose sets $S_{\ell}\subseteq Q_{\ell}$ such that $|S_{\ell}|=s$, for every $\ell\in \{0,1,\dots, k-2\}$. Let $b_1$ denote the number of such choices.
\item Next we choose the graph $G'$ on $[n]\backslash \bigcup_{\ell\in \{0,1,\dots, k-2\}} S_{\ell}$ such that $G[[n]\backslash \bigcup_{\ell\in \{0,1,\dots, k-2\}} S_{\ell}]=G'$. Let $b_2$ denote the number of possibilities for $G'$.
\item Next we choose the set $E'$ of internal edges of $G$ that are incident to at least one vertex in $\bigcup_{\ell\in \{0,1,\dots, k-2\}} S_{\ell}$ in such a way that the resulting graph $G$ will satisfy $S_{\ell}\subseteq Y_{\ell}(G,Q)$ for every $\ell\in \{0,1,\dots, k-2\}$. Let $b_3$ denote the number of possibilities for $E'$.
\item Finally we choose the set $E''$ of crossing edges of $G$ that are incident to at least one vertex in $\bigcup_{\ell\in \{0,1,\dots, k-2\}} S_{\ell}$. Let $b_4$ denote the number of possibilities for $E''$.
\end{itemize}
Hence,
\begin{equation}\label{eqn: F2 count}
|F^2_Q|\leq b_1\cdot b_2\cdot b_3\cdot b_4.
\end{equation}
The following series of claims will give upper bounds for the quantities $b_1,\dots, b_4$. The proof of Claim~2 is almost identical to that of Claim~2 in Lemma~\ref{F^1}; we give proofs of the others.

\vspace{0.3cm}
\noindent {\bf Claim 1:} $b_1\leq 2^n$.

\noindent Indeed,
$$b_1 \leq {\binom{n}{\left\lceil \frac{10n}{\log n}\right\rceil}}^{k-1} \leq \left( \left(\frac{e\log n}{10}\right)^{\frac{10n}{\log n}}\right)^{k-1} \leq 2^n,$$
as required.

\vspace{0.3cm}
\noindent {\bf Claim 2:} $b_2\leq C 2^{\mu^{1/2}n} f_k(\lceil  n/(k-1)-s  \rceil ) 2^{t_{k-1}(n-(k-1)s)}$.

\COMMENT{Indeed, note that for every graph $\tilde{G}\in F^2_Q$, Corollary~\ref{new partition} together with (\ref{size mathcal{P}}) implies that every optimal ordered $(k-1)$-partition of $\tilde{G}[[n]\backslash \bigcup_{\ell\in \{0,1,\dots, k-2\}} S_{\ell}]$ is contained in some set $\mathcal{P}$ of size at most $2^{\mu n}$. Since $G[[n]\backslash \bigcup_{\ell\in \{0,1,\dots, k-2\}} S_{\ell}]$ is clearly induced-$C_{2k}$-free, this together with (\ref{induct}) implies that
\begin{align*}
b_2 &\leq \sum_{Q'\in \mathcal{P}} |F_{Q'}(n- (k-1)s,k)| \leq 6C 2^{\mu n}2^{6(\log n)^2} f_k(\lceil n/(k-1)-s \rceil ) 2^{t_{k-1}(n-(k-1)s)}\\
&\leq C 2^{\mu^{1/2}n} f_k(\lceil  n/(k-1)-s  \rceil ) 2^{t_{k-1}(n-(k-1)s)},
\end{align*}
as required.}

\vspace{0.3cm}
\noindent {\bf Claim 3:} $b_3\leq 2^n$.

\noindent Indeed, for every graph $G^*\in F^2_Q$ for which $S_{\ell}\subseteq Y_{\ell}(G^*,Q)$ for every $\ell\in \{0,1,\dots, k-2\}$, let $G^*_{B, \ell}:= \overline{G^*}[Q_{\ell}\backslash A^{\ell}_{G^*}]$. Then each $S_{\ell}$ consists of isolated vertices in $G^*_{B, \ell}$ as well as non-centre vertices of a single component $\tilde{C}$ of $G^*_{B, \ell}$. (Note that $\tilde{C}$ is a star or triangle in $G^*_{B, \ell}$, with some centre $u\in Q_{\ell}\setminus (A^{\ell}_{G^*}\cup S_{\ell})$.) By Lemma~\ref{size of A^i_G}, we also have that $|A^{\ell}_{G^*}|\leq M$.

Hence $b_3\leq \prod_{\ell=0}^{k-2} \prod_{j=1}^3 h(\ell,j)$, where the quantities $h(\ell, j)$ are defined as follows. Let $h(\ell,1)$ denote the number of ways to choose a set $\tilde{A}^{\ell}\subseteq Q_{\ell}\backslash S_{\ell}$ of size at most $M$. (In what follows $\tilde{A}^{\ell}$ will play the role of $A^i_G$.) Then $h(\ell,1)\leq n^M$. Given such a set $\tilde{A}^{\ell}$, let $h(\ell,2)$ denote the number of ways to choose $\tilde{C}$. Then $h(\ell, 2)\leq n 2^{|S_{\ell}|}+n^3\leq n 2^{s+1}$. (Indeed, if $\tilde{C}$ is a star we have at most $n$ choices for the centre $u$, and for every vertex $v\in S_{\ell}$ we can choose whether $v$ is adjacent to $u$ or not; if $\tilde{C}$ is a triangle we have at most $n^3$ choices for its vertices.)
Given a set $\tilde{A}^{\ell}$ as above, let $h(\ell,3)$ denote the number of possible sets of edges between $S_{\ell}$ and $\tilde{A}^{\ell}$. Then $h(\ell,3)\leq 2^{|S_{\ell}||\tilde{A}^{\ell}|}\leq 2^{sM}$. Hence
$$b_3\leq \prod\limits_{\ell=0}^{k-2} \prod\limits_{j=1}^3 h(\ell,j)\leq (n^M\cdot n 2^{s+1}\cdot 2^{sM})^{k-1}\leq 2^n,$$
as required.

\vspace{0.3cm}
\noindent {\bf Claim 4:} $b_4\leq 2^{(k-2)sn-\binom{k-1}{2}s^2}$.

\noindent Indeed, note that, for a fixed index $\ell\in \{0,1,\dots, k-2\}$, the number $h_{\ell}$ of possible crossing edges in $G$ that are incident to a vertex in $S_{\ell}$ is at most $s(n - |Q_\ell|)$. Also, the number of possible crossing edges in $G$ that are incident to two vertices in $\bigcup_{\ell\in \{0,1,\dots, k-2\}} S_{\ell}$ is exactly $\binom{k-1}{2}s^2$. Hence,
$$b_4 \leq 2^{\sum_{\ell=0}^{k-2} h_{\ell}} 2^{-\binom{k-1}{2}s^2}\leq 2^{(k-2)sn-\binom{k-1}{2}s^2},$$
as required.

\vspace{0.3cm}
\noindent Note that $t_{k-1}(s(k-1)) = \binom{k-1}{2}s^2$ and that by Lemma \ref{f-estimate-2}, $f_k(n_k) \geq s^{s/2}f_k(\lceil  n/(k-1)-s  \rceil ) \geq 2^{4n} f_k(\lceil  n/(k-1)-s  \rceil )$. These observations together with (\ref{eqn: F2 count}), Claims~1--4 and Proposition~\ref{omitted proposition}(ii) imply that
\begin{align*}
|F^2_Q|&\leq 2^n \cdot C 2^{\mu^{1/2}n} f_k(\lceil  n/(k-1)-s  \rceil ) 2^{t_{k-1}(n-(k-1)s)} \cdot 2^n \cdot 2^{(k-2)sn-\binom{k-1}{2}s^2}\\
&\leq C 2^{3n}2^{-4n}f_k(n_k)2^{t_{k-1}(n-(k-1)s)+(k-2)sn-s(k-1)(k-2)-t_{k-1}(s(k-1))}\\
&\leq C 2^{-n} f_k(n_k) 2^{t_{k-1}(n)},
\end{align*}
as required.
\end{proof}

\section{Estimation of $|F^3_Q|$}\label{sec: estimate 3}

The information we have gained so far allows us to easily deduce that every $G\in F^3_Q$ is extremely close to being a $k$-template (see Proposition~\ref{beta}). One advantage of this is that it allows us to use more precise estimates when applying induction (see Corollary~\ref{new partition 2}).

\begin{proposition}\label{beta}
Let $G\in F^3_Q$ and $i\in \{0,1,\dots, k-2\}$. Then the following hold.
\begin{enumerate}[{\rm (i)}]
\item $\overline{G}[Q_i]$ is a disjoint union of stars and triangles.
\item $G$ contains at most $n$ internal non-edges.
\item Every vertex $x\in Q_i$ satisfies $\overline{d}^i_{G,Q}(x)< \beta n$.
\end{enumerate}
\end{proposition}
\begin{proof}
\begin{enumerate}[{\rm (i)}]
\item Since $G\notin F^2_Q$, every vertex $x \in Q_i$ satisfies $\min \{ d^i_{G,Q}(x),\overline{d}^i_{G,Q}(x)\} < \beta n$. Thus $A^i_G=\emptyset$, and so by Proposition~\ref{Q_i setminus A^i_G}, $\overline{G}[Q_i]$ is a disjoint union of stars and triangles.
\item This follows immediately from (i).
\item Let $x\in Q_i$. Let us first show that $d^i_{G,Q}(x) \geq \beta n$. Suppose not. Then $\overline{d}^i_{G,Q}(x) = |Q_i|-d^i_{G,Q}(x)-1 > |Q_i| -\beta n -1 \geq n/(k-1)-\mu n - \beta n -1$. Thus for every $j\in \{0,1,\dots, k-2\}$ with $j\ne i$, Proposition~\ref{optimality} implies that
\begin{align*}
d^j_{G,Q}(x) &= |Q_j| - \overline{d}^j_{G,Q}(x) \leq |Q_j| - \overline{d}^i_{G,Q}(x) < \left(\frac{n}{k-1} + \mu n \right) - \left(\frac{n}{k-1}-\mu n - \beta n -1 \right)\\
&= \beta n + 2\mu n +1 < \alpha n,
\end{align*}
where the last inequality follows from the fact that $\mu, \beta \ll \alpha$. Thus $x$ is both $i$-light and $j$-light, which contradicts the fact that $G\notin F^1_Q$. Thus $d^i_{G,Q}(x) \geq \beta n$. This together with the fact (observed in the proof of (i), above) that $A^i_G=\emptyset$ implies that $\overline{d}^i_{G,Q}(x) < \beta n$, as required.
\end{enumerate}
\end{proof}

Recall the definition of property $({\rm F}1)_{\nu}$ in Section~\ref{sec: set-up}. We define $T^*_Q(n,k) \subseteq F^3_Q$ to be the set of all (labelled) graphs in $F^3_Q$ that satisfy property $({\rm F}1)_{(40n\log n)^{1/2}/n}$ with respect to $Q$. Proposition~\ref{beta}(ii) together with Lemma~\ref{pre regular lemma}(i) applied with $(40n\log n)^{1/2}/n,n$ playing the roles of $\nu, m$ respectively implies that

\begin{equation}\label{strong regular}
|F^3_Q\setminus T^*_Q(n,k)| \leq 2^{t_{k-1}(n) - n\log n/5 }.
\end{equation}

So (\ref{strong regular}) allows us to restrict our attention to the class $T^*_Q(n,k)$. In particular, this allows us to apply property $({\rm F}1)_{\nu}$ to much smaller vertex sets than in the preceding sections. This in turn gives us a much better bound on the number of partitions that may arise after deleting a small number of vertices. More precisely, Lemma~\ref{pre new partition} applied with $(40n\log n)^{1/2}/n,n$ playing the roles of $\nu, m$ respectively implies the following result. Recall that $\mathcal{P}(Q,s)$ was defined before (\ref{size mathcal{P}}).

\begin{corollary}\label{new partition 2}
Let $S\subseteq [n]$ with $|S|\leq n/k^2$. Then for every $G\in T^*_Q(n,k)$, every optimal ordered $(k-1)$-partition of $G - S$ is an element of $\mathcal{P}(Q-S,40 k^4 \log n)$.
\end{corollary}

In order to estimate $|T_Q^*(n,k)|$ (and thus $|F^3_Q|$) we will further split $T_Q^*(n,k)$ into four classes $\mathcal{A}_1,\dots, \mathcal{A}_4$. To define these classes we require some further notation. We say that $G$ contains a {\em $(6,3)$-forest} with respect to $Q$ if there exist distinct indices $i,j\in \{0,1,\dots, k-2\}$ such that there exist six vertices in $Q_i\cup Q_j$ that induce on $G$ a linear forest with at most three components. A $(6,3)$-forest has potential extensions into an induced $C_{2k}$, so its existence in every $G\in \mathcal{A}_3$ (see below) constrains the possible edge sets for $G$ (and thus the number of choices for $G$). To obtain a significant constraint on the possible edge sets however, we first need to exclude the situations that arise in the classes $\mathcal{A}_1$ and $\mathcal{A}_2$, described below. These involve the structure of the stars of the complement graph inside the vertex classes, so to describe these classes of graphs recall that the centres of stars and triangles were defined before Proposition~\ref{prop: not j j' identical}. Given a graph $G\in F^3_Q$ and an index $i\in \{0,1,\dots, k-2\}$ we define the following sets.

\begin{itemize}
\item $C^i(G,Q)$ is the set of all centres of triangles and non-trivial stars in $\overline{G}[Q_i]$.

\item $C^i_{high}(G,Q)$ is the set of all centres of stars in $\overline{G}[Q_i]$ of order at least $n^{1-\frac{1}{2 k^2}}/200k^2$.

\item $B^i_{high}(G,Q)$ is the set of all vertices in $Q_i$ which have a non-neighbour in $C^i_{high}$.

\item $C^i_{low}(G,Q)$ is the set of all centres of triangles and non-trivial stars in $\overline{G}[Q_i]$ of order less than $n^{1-\frac{1}{2 k^2}}/200k^2$.

\item $B^i_{low}(G,Q)$ is the set of all vertices in $Q_i$ which have a non-neighbour in $C^i_{low}$.

\item $C^i_{0}(G,Q)$ is the set of all isolated vertices in $\overline{G}[Q_i]$.
\end{itemize}

We may sometimes write $C^i$ for $C^i(G,Q)$ when the graph $G$ and ordered $(k-1)$-partition $Q$ we consider are clear from the context (and similarly for $C^i_{high}, B^i_{high}, C^i_{low}, B^i_{low}, C^i_0$). Note that Proposition~\ref{beta}(i) implies that $C^i_{high}, B^i_{high}, C^i_{low}, B^i_{low}, C^i_0$ form a partition of $Q_i$. Given a subset $B\subseteq B^i_{low}$, we denote by $C(B)$ the set of all vertices in $C^i_{low}$ that have a non-neighbour in $B$.

We partition $T^*_Q(n,k)$ into the sets $\mathcal{A}_1, \dots, \mathcal{A}_4$ defined as follows.

\begin{itemize}
\item $\mathcal{A}_1$ is the set of all graphs $G \in T^*_Q(n,k)$ for which there exist distinct indices $i,j\in \{0,1,\dots, k-2\}$ such that $|B^{i}_{low}| \geq n/2k^2$ and there exist distinct vertices $y_1,y_2,y_3 \in Q_j$ that satisfy $|\overline{N}(\{y_1,y_2,y_3\}) \cap B^{i}_{low} | \leq n/200 k^2$.

\item $\mathcal{A}_{2}$ is the set of all graphs $G\in T^*_Q(n,k) \setminus \mathcal{A}_1$ for which there exist distinct indices $i,j\in \{0,1,\dots, k-2\}$ such that $|B^{i}_{low}| \geq n/2k^2$ and there exist distinct vertices $y_1,y_2,y_3 \in Q_j$ with $y_1,y_2 \notin C^j(G,Q)$ that satisfy
\begin{equation}\label{picture equation}
C(\overline{N}(\{y_1,y_2,y_3\}) \cap B^{i}_{low}) \cap \overline{N}(\{y_1,y_2\}) = \emptyset.
\end{equation}
(See Figure $1$.)

\begin{figure}
\centering
\begin{tikzpicture}[scale=0.6]
	\draw[line width=0.8pt] (0,0) ellipse (2 and 3);
	\draw[line width=0.8pt] (5,0) ellipse (2 and 3);
	\draw[line width=0.8pt] (1,0) ellipse (0.8 and 2);
	\draw[line width=0.8pt] (-0.6,1) ellipse (0.6 and 1);
	\filldraw[fill=black] (1,1.5) circle (2pt);
	\filldraw[fill=black] (1,0.5) circle (2pt);
	\filldraw[fill=black] (1,-0.5) circle (2pt);
	\filldraw[fill=black] (1,-1.5) circle (2pt);
	\filldraw[fill=black] (-0.6,1.5) circle (2pt);
	\filldraw[fill=black] (5,1.5) circle (2pt);
	\filldraw[fill=black] (5,0.5) circle (2pt);
	\filldraw[fill=black] (5,-0.5) circle (2pt);
	\filldraw[fill=black] (-0.6,-0.5) circle (2pt);
	\draw (1,-1.5) -- (-0.6,-0.5) -- (1,-0.5) -- (5,-0.5) -- (1,1.5) -- (5,0.5) -- (1,0.5) -- (-0.6,1.5) -- (1,1.5) -- (5,1.5);
	\draw (-0.6,1.5) to[bend left] (5,1.5);
	\node at (5.5,1.5) {$y_1$};
	\node at (5.5,0.5) {$y_2$};
	\node at (5.5,-0.5) {$y_3$};
	\node at (-0.6,1.1) {$x$};
	\node at (0,-3.5) {$Q_i$};
	\node at (5,-3.5) {$Q_j$};
	\node at (2.2,-2.4) {$B^i_{low}$};
	\node at (1.6,3.6) {$C(\overline{N}(\{y_1,y_2,y_3\})\cap B^i_{low})$};
	\draw[line width=0.2pt, ->, >= angle 60] (1.4, 3.2) -- (-0.4,2);
	\draw[line width=0.2pt, ->, >= angle 60] (1.6, -2.4) -- (1.2,-2);
	\node at (2.5,-4.2) {\textit{Figure 1: An illustration of $\overline{G}$ for $G\in \mathcal{A}_2$. Note that}};
	\node at (2.5,-4.9) {\textit{{\rm (\ref{picture equation})} implies that at most one of $xy_1, xy_2$ is an edge in $\overline{G}$.}};
\end{tikzpicture}
\end{figure}

\item $\mathcal{A}_3$ is the set of all graphs $G \in T^*_Q(n,k) \setminus (\mathcal{A}_1 \cup \mathcal{A}_2)$ such that $G$ contains a $(6,3)$-forest with respect to $Q$.

\item $\mathcal{A}_4 := T^*_Q(n,k) \setminus (\mathcal{A}_1 \cup \mathcal{A}_2 \cup \mathcal{A}_3)$ is the set of all remaining graphs.
\end{itemize}

We will estimate the sizes of $\mathcal{A}_1,\dots,\mathcal{A}_4$ separately. Lemma~\ref{lem: A_1} below gives a bound on $|\mathcal{A}_1|$. The idea of the proof of Lemma~\ref{lem: A_1} is that in this case the neighbourhoods of $y_1,y_2,y_3$ are `atypical', and hence a Chernoff estimate (see Claim~4) shows that graphs in $\mathcal{A}_1$ are rare.

\begin{lemma}\label{lem: A_1}
$|\mathcal{A}_1| \leq C 2^{-n/150k^2} f_k(n_k) 2^{t_{k-1}(n)}$.
\end{lemma}
\begin{proof}
Any graph $G\in \mathcal{A}_1$ can be constructed as follows.
\begin{itemize}
\item First we choose distinct indices $i,j\in \{0,1,\dots, k-2\}$, distinct vertices $y_1,y_2,y_3\in Q_j$, and a set $E$ of edges between $y_1,y_2,y_3$. Let $b_1$ denote the number of such choices. The choices in the next steps will be made such that $G$ satisfies $|B^{i}_{low}| \geq n/2k^2$ and $|\overline{N}(\{y_1,y_2,y_3\}) \cap B^{i}_{low} | \leq n/200 k^2$.
\item Next we choose the graph $G'$ on vertex set $[n]\backslash \{y_1,y_2,y_3\}$ such that $G[[n]\backslash \{y_1,y_2,y_3\}]=G'$. Let $b_2$ denote the number of possibilities for $G'$.
\item Next we choose the set $E'$ of edges in $G$ between $\{y_1,y_2,y_3\}$ and $Q_j\backslash \{y_1,y_2,y_3\}$. Let $b_3$ denote the number of possibilities for $E'$.
\item Finally we choose the set $E''$ of edges in $G$ between $\{y_1,y_2,y_3\}$ and $[n]\backslash Q_j$ such that $E''$ is compatible with our previous choices. Let $b_4$ denote the number of possibilities for $E''$.
\end{itemize}
Hence,
\begin{equation}\label{eqn: A_1 count}
|\mathcal{A}_1|\leq b_1\cdot b_2\cdot b_3\cdot b_4.
\end{equation}
The following series of claims will give upper bounds for the quantities $b_1,\dots, b_4$. Claim~1 is trivial; we give proofs of the others.

\vspace{0.3cm}
\noindent {\bf Claim 1:} $b_1\leq 2^3 k^2 n^3$.

\vspace{0.3cm}
\noindent {\bf Claim 2:} $b_2\leq C 2^{2(\log n)^3} f_k(n_k) 2^{t_{k-1}(n-3)}$.

\noindent Indeed, note that for every graph $\tilde{G}\in \mathcal{A}_1$, Corollary~\ref{new partition 2} together with (\ref{size mathcal{P}}) implies that every optimal ordered $(k-1)$-partition of $\tilde{G}[[n]\backslash \{y_1,y_2,y_3\}]$ is contained in some set $\mathcal{P}$ of size at most $2^{(\log n)^3}$. Since $G[[n]\backslash \{y_1,y_2,y_3\}]$ is clearly induced-$C_{2k}$-free, this together with (\ref{induct}) implies that
\begin{align*}
b_2 &\leq \sum_{Q'\in \mathcal{P}} |F_{Q'}(n-3,k)| \leq 6C 2^{(\log n)^3}2^{6(\log n)^2} f_k(\lceil (n-3)/(k-1) \rceil ) 2^{t_{k-1}(n-3)}\\
&\leq C 2^{2(\log n)^3} f_k(n_k) 2^{t_{k-1}(n-3)},
\end{align*}
as required.

\vspace{0.3cm}
\noindent {\bf Claim 3:} $b_3\leq 2^{3\xi(\beta) n}$.

\noindent Indeed, for every graph $\tilde{G}\in \mathcal{A}_1$ and every $\ell \in [3]$, Proposition~\ref{beta}(iii) implies that $\overline{d}^j_{\tilde{G},Q}(y_{\ell})< \beta n$.
Thus
$$b_3\leq \binom{n}{\leq \beta n} ^3\leq 2^{3\xi(\beta) n},$$
as required.

\vspace{0.3cm}
\noindent {\bf Claim 4:} $b_4\leq 2^{3((k-2)n/(k-1)+\mu n)-n/128k^2}$.

\noindent Consider the graph obtained by starting with the graph $([n], E(G')\cup E')$ and adding edges between $\{y_1,y_2,y_3\}$ and $[n]\backslash Q_j$ randomly, independently, with probability $1/2$. Note that the number of graphs that this process can generate is at most $2^{3((k-2)n/(k-1)+\mu n)}$, with each such graph equally likely to be generated. So an upper bound on $b_4$ is given by
$$b_4\leq 2^{3((k-2)n/(k-1)+\mu n)}\mathbb{P}\left(|\overline{N}(\{y_1,y_2,y_3\})\cap B^i_{low}| \leq \frac{n}{200k^2} \right).$$
Since $G'$ was chosen such that $|B^i_{low}|\geq n/2k^2$, we have that $\mathbb{E}(|\overline{N}(\{y_1,y_2,y_3\})\cap B^i_{low}|)\geq n/16k^2$. So Lemma~\ref{Chernoff Bounds}(ii) implies that
$$\mathbb{P}\left(|\overline{N}(\{y_1,y_2,y_3\})\cap B^i_{low}| \leq \frac{n}{200k^2} \right) \leq \exp \left( - \frac{n}{128k^2} \right)\leq 2^{- \frac{n}{128k^2}}.$$
Hence $b_4\leq 2^{3((k-2)n/(k-1)+\mu n)-n/128k^2}$, as required.

\vspace{0.3cm}
\noindent Now (\ref{eqn: A_1 count}) together with Claims~1--4 and Proposition~\ref{omitted proposition}(ii) implies that
\begin{align*}
|\mathcal{A}_1|&\leq 2^3 k^2 n^3 \cdot C 2^{2(\log n)^3} f_k(n_k) 2^{t_{k-1}(n-3)} \cdot 2^{3\xi(\beta) n} \cdot 2^{3((k-2)n/(k-1)+\mu n)-n/128k^2}\\
&\leq C 2^{-n/150k^2} f_k(n_k) 2^{t_{k-1}(n-3)+3(k-2)n/(k-1) - 3(k-2) - 3}\\
&\leq C 2^{-n/150k^2} f_k(n_k) 2^{t_{k-1}(n)},
\end{align*}
as required.
\end{proof}

\begin{lemma}\label{lem: A_2}
$|\mathcal{A}_2| \leq C 2^{-n^{1/2k^2}/3} f_k(n_k) 2^{t_{k-1}(n)}$.
\end{lemma}
\begin{proof}
Note that for every $G\in \mathcal{A}_2$ and every $s\in \{0,1,\dots, k-2\}$ the definition of $C^s(G,Q)$ implies that $Q_s\backslash C^s(G,Q)\geq |Q_s|/2$. So any graph $G\in \mathcal{A}_2$ can be constructed as follows. We first choose $a\in \mathbb{N}$ such that $n/2k^2 \leq a\leq n$, and then perform the following steps.
\begin{itemize}
\item We choose distinct indices $i,j\in \{0,1,\dots, k-2\}$, a set $W=\{y_1,y_2\}\cup \{w_{\ell}^s: \ell\in [2], s\in \{0,1,\dots, k-2\}\backslash \{j\} \}$ of vertices satisfying $y_1,y_2\in Q_j$ and $w_1^s, w_2^s\in Q_s$ for every $s\in \{0,1,\dots, k-2\}\backslash \{j\}$, a vertex $y_3\in Q_j\backslash W$, and a set $E$ of edges between the vertices in $W\cup \{y_3\}$. Let $b_1$ denote the number of such choices. The choices in this step and the next steps will be made such that $y_1,y_2\notin C^j(G,Q)$ and $w_1^s, w_2^s\notin C^s(G,Q)$ for every $s\in \{0,1,\dots, k-2\}\backslash \{j\}$, and $|B^{i}_{low}|=a$ and $C(Y) \cap \overline{N}(\{y_1,y_2\}) = \emptyset$, where $Y:= \overline{N}(\{y_1,y_2,y_3\}) \cap B^{i}_{low}(G,Q)$.
\item Next we choose the graph $G'$ on vertex set $[n]\backslash W$ such that $G[[n]\backslash W]=G'$. Let $b_2$ denote the number of possibilities for $G'$.
\item Next we choose the set $E'$ of internal edges in $G$ with exactly one endpoint in $W$ such that $E'$ is compatible with our previous choices\COMMENT{Note that we are overcounting here, since we have already fixed the edges between $\{y_1, y_2\}$ and $y_3$, but this is fine.}. Let $b_3$ denote the number of possibilities for $E'$.
\item Next we choose the set $E''$ of crossing edges in $G$ between $W$ and $B^{i}_{low}\backslash W$ such that $E''$ is compatible with our previous choices. Let $b_4$ denote the number of possibilities for $E''$.
\item Finally we choose the set $E'''$ of crossing edges in $G$ between $W$ and $[n]\backslash (W\cup B^{i}_{low})$ such that $E'''$ is compatible with our previous choices. Let $b_5$ denote the number of possibilities for $E'''$.
\end{itemize}
Hence
\begin{equation}\label{eqn: A_2 count}
|\mathcal{A}_2|\leq n\max\limits_{n/2k^2 \leq a\leq n} \{b_1\cdot b_2\cdot b_3\cdot b_4\cdot b_5 \}.
\end{equation}

The main idea of the proof is that since $Y$ is large for $G\in \mathcal{A}_2$, it follows that $C(Y)$ is also large. So the assumption that every element of $C(Y)$ has at least one neighbour in $\{y_1,y_2\}$ places a significant restriction on the number of choices for $G$. The role of the $w^s_{\ell}$ is to `balance out' the vertex classes, i.e. in the proof of Claim~5 it will be useful that $W$ contains two vertices from each vertex class.

The following series of claims will give upper bounds for the quantities $b_1,\dots, b_5$. Claims~1 and~4 are trivial, and the proof of Claim~2 proceeds in an almost identical way to that of Claim~2 in the proof of Lemma~\ref{lem: A_1}; we give proofs of Claims~3 and~5.

\vspace{0.3cm}
\noindent {\bf Claim 1:} $b_1\leq k^2 n^{2k-1} 2^{\binom{2k-1}{2}}$.

\vspace{0.3cm}
\noindent {\bf Claim 2:} $b_2\leq C 2^{2(\log n)^3} f_k(n_k) 2^{t_{k-1}(n-(2k-2))}$.

\vspace{0.3cm}
\noindent {\bf Claim 3:} $b_3\leq n^{4(k-1)}$.

\noindent Indeed, for every graph $\tilde{G}\in \mathcal{A}_2$ such that $y_1,y_2\notin C^j(\tilde{G},Q)$ and $w_1^s, w_2^s\notin C^s(\tilde{G},Q)$ for every $s\in \{0,1,\dots, k-2\}\backslash \{j\}$, Proposition~\ref{beta}(i) implies that $\overline{d}^j_{\tilde{G},Q}(y_{\ell}), \overline{d}^s_{\tilde{G},Q}(w^s_{\ell}) \leq 2$ for every $\ell \in [2]$ and every $s\in \{0,1,\dots, k-2\}\backslash \{j\}$.
Thus
$$b_3\leq n^{2|W|}\leq n^{4(k-1)},$$
as required.

\vspace{0.3cm}
\noindent {\bf Claim 4:} $b_4\leq 2^{(2k-4)a}$.

\vspace{0.3cm}
\noindent {\bf Claim 5:} $b_5\leq 2^{(2k-4)\left(n-a \right)} 2^{-2n^{1/2k^2}/5}$.

\noindent Indeed, suppose $G$ satisfies $C(Y) \cap \overline{N}(\{y_1,y_2\}) = \emptyset$. Since we choose $G$ such that $|B^i_{low}|=a\geq n/2k^2$, the fact that $G\notin \mathcal{A}_1$ implies that $|Y|>n/200k^2$. Now the definitions of $C^i_{low}, B^i_{low}$ imply that
$$|C(Y)|\geq \frac{200k^2|Y|}{n^{1-1/2k^2}}\geq n^{1/2k^2}.$$
So since in $G$ every vertex in $C(Y)$ must have at least one neighbour in $\{y_1,y_2\}$,
\begin{align}\label{eqn: b_5}
b_5&\leq  2^{2\sum_{s\in \{0,1,\dots, k-2\}\backslash \{j\}} |[n]\backslash (Q_s\cup B^i_{low})|} 2^{2|[n]\backslash (Q_j\cup B^i_{low}\cup C(Y))|}3^{|C(Y)|}\\
&\leq 2^{(2k-4)\left(n-a \right)} 2^{-2n^{1/2k^2}/5},\nonumber
\end{align}
as required. The second inequality of (\ref{eqn: b_5}) is where it is important that $W$ contains two vertices from each vertex class.

\vspace{0.3cm}
\noindent Now (\ref{eqn: A_2 count}) together with Claims~1--5 and Proposition~\ref{omitted proposition}(ii) implies that
\begin{align*}
|\mathcal{A}_2|\leq \hspace{0.1cm} &n\cdot k^2 n^{2k-1} 2^{\binom{2k-1}{2}} \cdot C 2^{2(\log n)^3} f_k(n_k) 2^{t_{k-1}(n-(2k-2))}\\
&\cdot n^{4(k-1)}  \cdot \max\limits_{n/2k^2 \leq a\leq n} \left\{ 2^{(2k-4)a} \cdot 2^{(2k-4)\left(n-a \right)} 2^{-2n^{1/2k^2}/5} \right\}\\
\leq \hspace{0.1cm} &C 2^{-n^{1/2k^2}/3} f_k(n_k)\cdot 2^{t_{k-1}(n-(2k-2))+(2k-2)(k-2)n/(k-1) - (2k-2)(k-2) - t_{k-1}(2k-2)}\\
\leq \hspace{0.1cm} &C 2^{-n^{1/2k^2}/3} f_k(n_k) 2^{t_{k-1}(n)},
\end{align*}
as required.
\end{proof}

As mentioned earlier, a $(6,3)$-forest (with edge set $E$ say) is a useful building block for constructing many induced copies of $C_{2k}$. More precisely, in Lemma~\ref{lem: A_3} we will show that there are many `$E$-compatible' linear forests $H$, which play a similar role to that of the skeletons in the proof of Lemma~\ref{F^1}. Each such $E\cup E(H)$ gives us a non-trivial restriction on the remaining edge set, resulting in an adequate bound on $|\mathcal{A}_3|$.

\begin{lemma}\label{lem: A_3}
$|\mathcal{A}_3| \leq C 2^{-\frac{n}{2^{14k}}} f_k(n_k) 2^{t_{k-1}(n)}$.
\end{lemma}
\begin{proof}
Any graph $G\in \mathcal{A}_3$ can be constructed as follows.
\begin{itemize}
\item First we choose distinct indices $i,j\in \{0,1,\dots, k-2\}$, a set $X\subseteq Q_i\cup Q_j$ of six vertices, and a set $E$ of edges between vertices in $X$ such that the graph $(X,E)$ is a linear forest with at most three components (so $E$ will be the edge set of a $(6,3)$-forest in $G$). Let $b_1$ denote the number of such choices.
\item Next we choose a graph $G'$ on vertex set $[n]\backslash X$ such that $G[[n]\backslash X]=G'$. Let $b_2$ denote the number of possibilities for $G'$.
\item Next we choose the set $E'$ of internal edges in $G$ with exactly one endpoint in $X$. Let $b_3$ denote the number of possibilities for $E'$.
\item Finally we choose the set $E''$ of crossing edges in $G$ between $X$ and $[n]\backslash X$ such that $E''$ is compatible with our previous choices. Let $b_4$ denote the number of possibilities for $E''$.
\end{itemize}
Hence,
\begin{equation}\label{eqn: A_3 count}
|\mathcal{A}_3|\leq b_1\cdot b_2\cdot b_3\cdot b_4.
\end{equation}
The following series of claims will give upper bounds for the quantities $b_1,\dots, b_4$. Claim~1 is trivial, and the proofs of Claims~2 and~3 follow in an almost identical way to those of Claims~2 and~3 in the proof of Lemma~\ref{lem: A_1}, so we give only a proof of Claim~4.

\vspace{0.3cm}
\noindent {\bf Claim 1:} $b_1\leq 2^{15} k^2 n^6$.

\vspace{0.3cm}
\noindent {\bf Claim 2:} $b_2\leq C 2^{2(\log n)^3} f_k(n_k) 2^{t_{k-1}(n-6)}$.

\vspace{0.3cm}
\noindent {\bf Claim 3:} $b_3\leq 2^{6\xi(\beta) n}$.

\vspace{0.3cm}
\noindent {\bf Claim 4:} $b_4\leq 2^{\frac{6(k-2)n}{k-1}} 2^{\mu^{1/4}n} 2^{-\frac{n}{2^{13k}}}$.

\noindent Indeed, we define an \emph{$E$-compatible forest} to be a linear forest $H$ on $2k-6$ vertices, with the same number of components as $(X,E)$, such that $V(H)\cap Q_s$ induces a clique on two vertices for every $s\in \{0,1,\dots, k-2\}\backslash \{i,j\}$. 
Note that an $E$-compatible forest exists since $2k-6\geq 2 \cdot 3$ and $(X,E)$ has at most three components.
Moreover, an $E$-compatible forest contains a perfect matching, 
so Proposition~\ref{building} implies that for every graph $\tilde{G}\in \mathcal{A}_3$, the number of disjoint $E$-compatible forests in $\tilde{G}$ is at least
$$\left\lfloor \frac{n/(k-1)-\mu n - 2 \mu^{1/2} n}{2} \right\rfloor \geq \frac{n}{2(k-1)}-3\mu^{1/2} n.$$
Hence $G'$ contains at least $n/2(k-1)-3\mu^{1/2} n$ disjoint $E$-compatible forests. Now fix a set $CF$ of $n/2(k-1)-3\mu^{1/2} n$ disjoint $E$-compatible forests in $G'$, and let $H\in CF$. Let $h_H$ denote the number of possibilities for a set $E^*$ of edges between $X$ and $V(H)$. By Proposition~\ref{incomplete} there exists at least one set $\tilde{E}$ of edges between $X$ and $V(H)$ such that the graph $(X\cup V(H), E\cup E(H)\cup \tilde{E})$ is isomorphic to $C_{2k}$. So since $G$ must be induced-$C_{2k}$-free, we must have that $E^*\ne\tilde{E}$, and hence $h_H\leq 2^{|X||V(H)|}-1=2^{12(k-3)}-1$. Note that the number of vertices outside $Q_i\cup Q_j$ that are not contained in some graph $H\in CF$ is at most $(k-3)n/(k-1) + 2\mu n - (2k-6)(n/2(k-1)-3\mu^{1/2} n)\leq 6k\mu^{1/2} n$. Hence,
\begin{align*}
b_4&\leq 2^{6\cdot \max \{|Q_i|, |Q_j|\}} 2^{6(6k\mu^{1/2} n)} \prod\limits_{H\in CF} h_H\\
&\leq 2^{6(n/(k-1)+\mu n)} 2^{6(6k\mu^{1/2} n)} \left( 2^{12(k-3)}\left( 1-2^{-12(k-3)} \right) \right)^{n/(2(k-1))-3\mu^{1/2} n}\\
&\leq 2^{\frac{6(k-2)n}{k-1}} 2^{40k\mu^{1/2} n} e^{-\frac{n/(2(k-1))}{2^{12(k-3)}}} \leq 2^{\frac{6(k-2)n}{k-1}} 2^{\mu^{1/4}n} 2^{-\frac{n}{2^{13k}}},
\end{align*}
as required.

\vspace{0.3cm}
\noindent Now (\ref{eqn: A_3 count}) together with Claims~1--4 and Proposition~\ref{omitted proposition}(ii) implies that
\begin{align*}
|\mathcal{A}_3|&\leq 2^{15} k^2 n^6 \cdot C 2^{2(\log n)^3} f_k(n_k) 2^{t_{k-1}(n-6)} \cdot 2^{6\xi(\beta) n} \cdot 2^{\frac{6(k-2)n}{k-1}} 2^{\mu^{1/4}n} 2^{-\frac{n}{2^{13k}}}\\
&\leq C 2^{-\frac{n}{2^{14k}}} f_k(n_k) 2^{t_{k-1}(n-6)+6(k-2)n/(k-1) - 6(k-2) - t_{k-1}(6)}\\
&\leq C 2^{-\frac{n}{2^{14k}}} f_k(n_k) 2^{t_{k-1}(n)},
\end{align*}
as required.
\end{proof}

The next proposition shows that for every $G \in \mathcal{A}_4$, the small stars and triangles in $\overline{G}[Q_0]$ do not cover too many vertices.

\begin{proposition} \label{B_low}
For every $G \in \mathcal{A}_4$ and index $i\in \{0,1,\dots, k-2\}$, $|B^i_{low}| < n/2k^2$.
\end{proposition}
\begin{proof}
Suppose for a contradiction that there exists a graph $G \in \mathcal{A}_4$ such that $|B^i_{low}| \geq n/2k^2$ for some index $i\in \{0,1,\dots, k-2\}$. Since $G \in \mathcal{A}_4 \subseteq F^3_Q$, $G$ is not a $k$-template. This fact together with Proposition~\ref{beta}(i) implies that there exists an index $j\in \{0,1,\dots, k-2\}\backslash \{i\}$ and a non-edge $y_1y_3$ inside $Q_j$. At most one of $y_1,y_3$ can be in $C^j$ (by definition of $C^j$), and so without loss of generality we assume that $y_1 \notin C^j$. So Proposition~\ref{beta}(i) implies that $\overline{d}^j_{G,Q}(y_1)\leq 2$. This together with the observation that $|Q_j\backslash C^j|\geq |Q_j|/2$ (by definition of $C^j$) implies that there exists a vertex $y_2 \in Q^j\setminus C^j$ that is a neighbour of $y_1$.

Define $Y:= \overline{N}(\{y_1,y_2,y_3\}) \cap B^{i}_{low}$. Since $|B^i_{low}| \geq n/2k^2$ and $G\notin \mathcal{A}_1$, $|Y|> n/200k^2$. Since $|B^i_{low}| \geq n/2k^2$ and $G\notin \mathcal{A}_2$, $C(Y)$ contains a vertex $x_3\in \overline{N}(\{y_1,y_2\})$. Since $x_3\in C(Y)$ there exists a vertex $x_1\in Y$ that is a non-neighbour of $x_3$. By Proposition~\ref{beta}(iii), $\overline{d}^i_{G,Q}(x_1), \overline{d}^i_{G,Q}(x_3) \leq \beta n$. So since $|Y| > n/200k^2 \geq 2\beta n$, there exists a vertex $x_2 \in Y\cap N(\{x_1,x_3\})$.

Then $E(G[\{x_1,x_2,x_3,y_1,y_2,y_3\}]) = \{ x_1x_2, x_2x_3, y_1y_2\} \cup E'$ with $E'\subseteq \{y_2y_3, y_3x_3\}$. Thus the set $\{x_1,x_2,x_3,y_1,y_2,y_3\}\subseteq Q_i\cup Q_j$ induces on $G$ a linear forest with at most three components, and so $G$ contains a $(6,3)$-forest with respect to $Q$. This contradicts the fact that $G\notin \mathcal{A}_3$, and hence completes the proof.
\end{proof}

We now have sufficient information about the set $\mathcal{A}_4$ of remaining graphs to count them directly (i.e. $\mathcal{A}_4$ is the only class for which we do not use induction in our estimates). In particular, we now know that in $\overline{G}$ every vertex class is the union of triangles and stars, where crucially the number of triangles and small stars is not too large (see Proposition~\ref{B_low}). This allows us to show by a direct counting argument that $|\mathcal{A}_4|$ is negligible.

\begin{lemma}\label{lem: A_4}
$|\mathcal{A}_4| \leq 2^{-\frac{n\log n}{3k^2}} f_k(n_k) 2^{t_{k-1}(n)}$.
\end{lemma}
\begin{proof}
Any graph $G\in \mathcal{A}_4$ can be constructed as follows.
\begin{itemize}
\item First we choose a partition of $Q_i$ into five sets, $C^i_h, B^i_h, C^i_\ell, B^i_\ell, C^i_z,$ for every $i\in \{0,1,\dots, k-2\}$. Let $b_1$ denote the number of such choices.
\item Next we choose the set $E$ of crossing edges in $G$ with respect to $Q$. Let $b_2$ denote the number of possibilities for $E$.
\item Finally we choose the set $E'$ of internal edges in $G$ with respect to $Q$ such that $G$ satisfies $C^i_h=C^i_{high}$, $B^i_h=B^i_{high}$, $C^i_\ell=C^i_{low}$, $B^i_\ell=B^i_{low}$, and $C^i_z=C^i_0$ for every $i\in \{0,1,\dots, k-2\}$. Let $b_3$ denote the number of possibilities for $E'$.
\end{itemize}
Hence
\begin{equation}\label{eqn: A_4 count}
|\mathcal{A}_4|\leq b_1\cdot b_2\cdot b_3.
\end{equation}
The following series of claims will give upper bounds for the quantities $b_1,b_2,b_3$. Claims~1 and~2 are trivial; we give only a proof of Claim~3.

\vspace{0.3cm}
\noindent {\bf Claim 1:} $b_1\leq 5^n$.

\vspace{0.3cm}
\noindent {\bf Claim 2:} $b_2\leq 2^{t_{k-1}(n)}$.

\vspace{0.3cm}
\noindent {\bf Claim 3:} $b_3\leq 2^{ \frac{(k-1/2)n\log n}{k^2}}$.

\noindent For any given $i\in \{0,1,\dots, k-2\}$ and any vertex $x\in B^i_{high}$, the number of possibilities for the unique non-neighbour of $x$ in $C^i_{high}$ (namely the centre of the star in $\overline{G}$ containing $x$) is $|C^i_{high}|$. Now consider $x\in B^i_{low}$. Then $x$ has a unique non-neighbour $y$ in $C^i_{low}$, and has the possibility of either being part of a triangle in $\overline{G}$ or a star in $\overline{G}$. Note also that $|B^i_{low}|< n/2k^2$ by Proposition~\ref{B_low}, and that by definition of $C^i_{high}$,
$$|C^i_{high}|\leq \frac{200k^2 n}{n^{1-1/2k^2}}\leq 200k^2 n^{1/2k^2}.$$
Hence,
\begin{align*}
b_3 &\leq \prod_{i=0}^{k-2} (2|C^i_{low}|)^{|B^i_{low}|}  |C^i_{high}|^{|B^i_{high}|} \leq \prod_{i=0}^{k-2} n^{\frac{n}{2k^2}} (200k^2)^{n} (n^{\frac{1}{2k^2}})^n = 2^{ \frac{(k-1)n\log n}{k^2}} (200k^2)^{n(k-1)}\\
&\leq 2^{\frac{(k-1/2)n\log n}{k^2}},
\end{align*}
as required.

\vspace{0.3cm}
\noindent Now (\ref{eqn: A_4 count}) together with Claims~1--3 and Lemma~\ref{f-estimate-1} implies that
\begin{align*}
|\mathcal{A}_4|&\leq 5^n \cdot 2^{t_{k-1}(n)} \cdot 2^{ \frac{(k-1/2)n\log n}{k^2}}\\
&\leq 5^n 2^{-\frac{n\log n}{2k^2}} 2^{n_k \log n_k - en_k \log \log n_k} 2^{en_k \log \log n_k} 2^{t_{k-1}(n)}\leq 2^{-\frac{n\log n}{3k^2}} f_k(n_k) 2^{t_{k-1}(n)},
\end{align*}
as required.
\end{proof}

Recall that $F^3_Q= (F^3_Q\backslash T^*_Q(n,k))\cup \mathcal{A}_1 \cup \mathcal{A}_2 \cup \mathcal{A}_3 \cup \mathcal{A}_4$. The following bound on $|F^3_Q|$ follows immediately from this observation together with (\ref{strong regular}) and Lemmas~\ref{lem: A_1},~\ref{lem: A_2},~\ref{lem: A_3} and~\ref{lem: A_4}.

\begin{lemma}\label{F^3}
$|F^3_Q|\leq 2C2^{-n^{\frac{1}{2k^2}}/3}f(n_k) 2^{t_{k-1}(n)}$.
\end{lemma}

\section{Proof of Lemma~\ref{induction conclusion}}\label{sec: final calculation}

\removelastskip\penalty55\medskip\noindent{\bf Proof of Lemma~\ref{induction conclusion}.}
Recall from Section~\ref{sec: main proof} that we prove Lemma~\ref{induction conclusion} by induction on $n$ and that we choose constants satisfying (\ref{eqn: hierarchy}). The fact that $1/C \ll 1/n_0, 1/k$ implies that the statement of Lemma~\ref{induction conclusion} holds for all $n\leq n_0$. So suppose that $n>n_0$ and that the statement holds for all $n'< n$. Then we obtain the bounds in Lemmas~\ref{F^1}, \ref{F^2_Q} and~\ref{F^3}. These bounds together with the fact that $F_Q(n,k,\eta,\mu) = T_Q \cup F^1_Q \cup F^2_Q \cup F^3_Q$ and $T_Q\subseteq T_Q(n,k)$ imply that
\begin{align*}
|F_Q(n,k,\eta,\mu) \setminus T_Q(n,k)| &\leq  C\left(  2^{-\beta^2 n/14^{k}}  + 2^{-n} + 2\cdot 2^{-n^{\frac{1}{2k^2}}/3} \right) f_k(n_k) 2^{t_{k-1}(n)}\\
&\leq  3C 2^{-n^{\frac{1}{2k^2}}/3} f_k(n_k) 2^{t_{k-1}(n)}.
\end{align*}
This together with Corollary~\ref{regular lemma} implies that
\begin{align}\label{eqn: F_Q(n,k,eta) size}
|F_Q(n,k,\eta)\setminus T_Q(n,k)| &\leq |F_Q(n,k,\eta)\setminus F_Q(n,k,\eta,\mu)|+|F_Q(n,k,\eta,\mu)\setminus T_Q(n,k)| \\
& \leq  \left(2^{- \frac{\mu^2n^2}{100}} + 3C 2^{-n^{\frac{1}{2k^2}}/3} f_k(n_k) \right)  2^{t_{k-1}(n)}\nonumber\\
&\leq 4C 2^{-n^{\frac{1}{2k^2}}/3} f_k(n_k) 2^{t_{k-1}(n)}.\nonumber
\end{align}
Note that Lemma~\ref{lem: rough structure} (applied with $\eta/2$ playing the role of $\eta$) together with (\ref{eqn: near templates}) implies that
\begin{equation} \label{rough}
|F(n,k)\setminus F(n,k,\eta)| \leq 2^{-\varepsilon n^2}|F(n,k,\eta)|.
\end{equation}
Let $\mathcal{Q}$ denote the set of all ordered $(k-1)$-partitions of $[n]$, and recall that our choice of $Q\in \mathcal{Q}$ was arbitrary. Now (\ref{eqn: F_Q(n,k,eta) size}) together with (\ref{rough}) and Lemma~\ref{size of template} implies that
\begin{align*} 
|F(n,k)\setminus F(n,k,\eta)| &\leq 2^{-\varepsilon n^2} \sum_{Q'\in \mathcal{Q}} \left(|F_{Q'}(n,k,\eta)\backslash T_{Q'}(n,k)|+|T_{Q'}(n,k)|\right)\\
&\leq 2^{-\varepsilon n^2} (k-1)^n \left(4C 2^{-n^{\frac{1}{2k^2}}/3} + 2^{6 (\log n)^2} \right)f_k(n_k) 2^{t_{k-1}(n)}\\
&\leq C 2^{-\varepsilon n^2 /2 } f_k(n_k) 2^{t_{k-1}(n)}.
\end{align*}
Now this together with (\ref{eqn: F_Q(n,k,eta) size}) implies that
\begin{align*}
|F_Q(n,k)| &\leq |F_Q(n,k,\eta)| + |F(n,k)\setminus F(n,k,\eta)|\\
&\leq |T_Q(n,k)| + |F_Q(n,k,\eta)\setminus T_Q(n,k)| + |F(n,k)\setminus F(n,k,\eta)|\\
&\leq |T_Q(n,k)|+ \left( 4\cdot 2^{-n^{\frac{1}{2k^2}}/3} + 2^{-\epsilon n^2 /2} \right) C f_k(n_k) 2^{t_{k-1}(n)}\\
&\leq |T_Q(n,k)|+ 5C 2^{-n^{\frac{1}{2k^2}}/3}f_k(n_k) 2^{t_{k-1}(n)},
\end{align*}
which completes the inductive step, and hence the proof.
\endproof

\section{Acknowledgement}

We are grateful to Mihyun Kang for helpful remarks on the number of $k$-templates.

\medskip

{\footnotesize \obeylines \parindent=0pt

School of Mathematics
University of Birmingham
Edgbaston
Birmingham
B15 2TT
UK
}
\begin{flushleft}
{\it{E-mail addresses}:}
{\rm{\{j.kim.3, d.kuhn, d.osthus, txt238\}@bham.ac.uk}}
\end{flushleft}

\end{document}